\documentclass[12pt]{article}

\usepackage{amsmath, amsfonts, amssymb}

\usepackage[utf8]{inputenc}
\usepackage[english]{babel}
\usepackage{todonotes}
\usepackage{comment}
\usepackage{enumitem}
\usepackage[left=2cm, right=3cm, top=2cm, bottom=2cm, bindingoffset=0cm, marginparwidth=3cm]{geometry}

\usepackage{hyperref}
\usepackage{float}
\usepackage[outdir=./]{epstopdf}

\usepackage{amsthm}

\usepackage{tikz}
\usetikzlibrary{arrows,calc}
\tikzset{
>=stealth',
}

\newcounter{theorems}
\newtheorem{prop}[theorems]{Proposition}
\newtheorem{thm}[theorems]{Theorem}
\newtheorem{lem}[theorems]{Lemma}

\newtheorem*{claim}{CLAIM}

\theoremstyle{definition}
\newtheorem{defin}[theorems]{Definition}
\theoremstyle{remark}

\newtheorem{rem}[theorems]{Remark}
\newtheorem*{rem*}{Remark}

\makeatletter
\def\blfootnote{\gdef\@thefnmark{}\@footnotetext}
\makeatother

\def\Om {\Omega }
\def\e {\varepsilon }
\def\phi {\varphi}
\def\be {\begin{equation}}
\def\ee {\end{equation}}
\def\bt {\begin{thm}}
\def\et {\end{thm}}

\def\RR {\mathbb R}
\def\ZZ {\mathbb Z}

\def\nn {\mathbb N}
\def\QQ {\mathbb Q}

\DeclareMathOperator{\dist}{dist}

\def\TT{{\mathbb{T}^2}}

\def\t{{\theta}}
\def\ddt{\frac{\partial}{\partial \t}}

\def\Cyf{\mathcal{C}_{(-1, 1)}}
\def\Cyc{\mathcal{C}_{[-1, 1]}}

\def\ddzeta{\frac{\partial}{\partial\zeta}}
\def\trho{\tilde{\rho}}

\begin{document}

\title{Attractors of direct products}
\author{Stanislav Minkov\footnote{Brook Institute of Electronic Control Machines, Moscow, Russia}, \; Ivan Shilin\footnote{National Research University Higher School of Economics, Moscow, Russia}\hphantom{1}\thanks{Corresponding author; e-mail: i.s.shilin@yandex.ru}}
\date{}

\maketitle

\begin{abstract}
For Milnor, statistical, and minimal attractors, we construct examples of smooth flows $\varphi$ on $S^2$ for which the attractor of the Cartesian square of~$\varphi$ is smaller than the Cartesian square of the attractor of~$\varphi$. In the example for the minimal attractors, the flow~$\varphi$ also has a global physical measure such that its square does not coincide with the global physical measure of the square of $\varphi$.  
\end{abstract}



\section{Introduction}

When constructing examples of dynamical systems with required properties, it is not uncommon to utilize, at least as a piece of the construction, direct products of systems in lower dimensions. It is tempting to think that the attractor of the direct product of two systems always coincides with the direct product of their attractors. Although this holds, indeed, for so-called maximal attractors (we give all necessary definitions in Section~\ref{sec:def}), this is not true for several other types of attractors, namely, for Milnor, statistical, and minimal attractors, and also for the supports of physical measures, when the latter exist.

This non-coincidence for Milnor attractors was first considered by P.~Ashwin and M.~Field in~\cite{AF}: they conjectured that for the product of two planar flows with attracting homoclinic saddle loops the Milnor attractor does not contain the whole product of the two loops and verified that this indeed happened in a model case. N.~Agarwal, A.~Rodrigues, and M.~Field~\cite{ARF} proved this conjecture and also generalized this result to the case of arbitrary attracting polycycles formed by hyperbolic saddles. 

We start, in Section~\ref{sec:bowen}, with presenting an example of this type of non-coincidence for the case of statistical attractors. We consider the square of the so-called modified Bowen example: it is a flow on $S^2$ with a monodromic biangle formed by the separatrices of a saddle and a saddle-node; the presence of a saddle-node and of two separatrix connections makes it a flow of codimension~$3$. V.~Kleptsyn in~\cite{K06} used this example to show that the minimal attractor of a smooth flow can be a proper subset of its statistical attractor. We show, in Theorem~\ref{thm:stat}, that the statistical attractor of the square of such a flow is smaller than the square of the statistical attractor of the flow itself, and that the same also holds for Milnor attractors.

In Section~\ref{sec:loop} we revisit the codimension~$1$ example with an attracting separatrix loop and give a short proof (inspired by the argument that will appear later in Section~\ref{sec:twombe}) of the fact that for the square of such a flow the Milnor attractor is smaller than the square of the attractor of the flow itself (Theorem~\ref{thm:mil}). This is, of course, a particular case of the result of~\cite{ARF}, but it admits a proof that is much less technical than the proof for the general case given by Agarwal, Rodrigues, and~Field.

Although the product of two flows is a flow of infinite codimension, it is interesting to find, for every type of attractor, the least codimension for the factor-flows that is needed to construct an example of non-coincidence. For Milnor attractors of flows on $S^2$, both factor-flows have to be of codimension at least~$1$, and in this sense the product of separatrix loops provides the simplest possible example. Indeed, generic flows on the sphere are Morse--Smale\footnote{Generic flows are Morse--Smale for any orientable two-dimensional surface and also the projective plane, Klein bottle, and torus with a cross-cap, with respect to any $C^r$-topology on the space of vector fields, see~\cite{P62} and~\cite{G78}. This holds for other non-orientable surfaces if $r = 1$ (which can be seen by combining the discussion in~\cite{G78} with Pugh's Closing lemma), but remains an open question  for $r > 1$.}, and one can check that for a product of two Morse--Smale flows or a Morse--Smale flow and a codimension~$1$ flow the Milnor attractor coincides with the product of Milnor attractors of the factors.
As for the statistical attractors, we prove in Section~\ref{sec:lune} that for flows on~$S^2$ one cannot construct an example of non-coincidence analogous to the example from Section~\ref{sec:bowen} by taking a product of two arbitrary flows of codimension at most two, see Theorem~\ref{thm:codim2}.

In Section~\ref{sec:twombe} we generalize the result of Section~\ref{sec:bowen} to the case of a product of two arbitrary modified Bowen examples (Theorem~\ref{thm:twombe}). It may be worth noting that the product flows we consider can be viewed as examples of synchronization\footnote{Synchronization is a process where some dynamical properties of two systems (or of two trajectories of one system) evolve in an almost identical way, usually due to some interaction between the systems.
In our case synchronization takes the following form. We have two systems such that generic orbits spend almost all the time near two points: $A, B$ for the first system and $\tilde{A}, \tilde{B}$ for the second, but a generic orbit of the product-system spends almost all the time near the diagonal points $(A, \tilde{A})$ and $(B, \tilde{B})$ and neglects the pairs $(A, \tilde{B})$ and $(B, \tilde{A})$.
Synchronization in a similar sense was discussed in~\cite{GK16}; more classical definitions of synchronization can be found in, e.g.,~\cite{Y}.} between two systems that happens in the absence of interaction. In the case of a square of a flow, one would think that the reason for synchronization is that the systems are identical. However, the result of Section~\ref{sec:twombe} shows that this synchronization originates not from the exact coincidence of the two flows, but rather from some dynamical features shared by them.

Finally, in Section~\ref{sec:srb} we construct an example of a flow, of infinite codimension this time, whose square demonstrates the same phenomenon for minimal attractors and the supports of physical measures (see Theorem~\ref{thm:srb}). In this construction we utilize the idea of oscillating measures that might be of independent interest. The question of whether there exists an example of non-coincidence in finite codimension remains open for the case of minimal attractors and physical measures.

Flows on $S^2$ seem to provide a natural starting point and the simplest setting for studying how abundant this type of non-coincidence can be, which we attempted to do.
It would be interesting to find out what can be said about generic diffeomorphisms and flows (in higher dimensions) with respect to this non-coincidence.

\subsection*{Acknowledgements} The idea of the paper was conceived in a discussion with Alexey Okunev, to whom we are very grateful. We also want to thank Andrey Dukov whose comments on the draft of this paper greatly helped us in improving it. We are deeply indebted to Yulij Ilyashenko for his heartening words and remarks. Both authors were partially supported by the RFBR grant 20-01-00420-a.

\section{Definitions of attractors}\label{sec:def}

To make this section brief, all the definitions are given for the case of continuous time. Our phase space~$X$ will always be a compact $C^\infty$-smooth manifold (sometimes with boundary) with some natural metric and with Lebesgue probability measure, which we denote~$\mu$, and our dynamical system will be a $C^\infty$-smooth flow\footnote{We do not distinguish flows and semi-flows. In each case it will be clear from the context whether the flow is invertible or not.} $\varphi = (\varphi^t)_t$.  

If a dynamical system has an absorbing domain (also called trapping region), i.e., an open set~$U \subset X$ such that $\varphi^t(\overline U) \subset U$ for positive~$t$, then \emph{the maximal attractor of $\varphi$ in~$U$} is the set $A_{max}(\varphi, U) = \cap_{t \ge 0} \varphi^t(U)$. This type of attractor is of little interest to us, however, because, if two systems have absorbing domains, the maximal attractor of the product-system in the product of the domains coincides with the product of maximal attractors; the proof is straightforward.
%HERE LIES THE PROOF FOR MAXIMAL ATTRACTORS (FOR DISCRETE TIME)
\begin{comment}
\begin{prop}
Let $U$ be an absorbing domain for a map (or flow)~$F$ and~$W$ be an absorbing domain for~$G$. Then $A_{max}(F \times G, U \times W) = A_{max}(F, U) \times A_{max}(G, W)$. 
\end{prop}
\begin{proof}
We present the proof for the discrete time case. If~$F$ and~$G$ are flows, replace all instances of~$F, G$ by~$F^t, G^t$.

Since $F(\overline{U}) \subset U$ and $G(\overline{W}) \subset W$, we have
$$(F \times G)(\overline{U \times W}) = (F \times G)(\overline{U} \times \overline{W}) = F(\overline{U}) \times G(\overline{W}) \subset U \times W,$$
so $U \times W$ is an absorbing domain and $A_{max}(F \times G, U \times W)$ is defined. The following equivalences hold:
$$(a, b) \in A_{max}(F \times G, U \times W) \Leftrightarrow
\forall n \in \nn \; (a, b) \in F^n(U) \times G^n(W) \Leftrightarrow$$
$$\Leftrightarrow a \in A_{max}(F, U) \mbox{ and } b \in A_{max}(G, W) \Leftrightarrow (a, b) \in A_{max}(F, U) \times A_{max}(G, W).$$
\end{proof}
\end{comment}

Transitive maximal attractors are called \emph{topological attractors}. The product of topological attractors may fail to be transitive and hence fail to be a topological attractor: e.g., consider a product of attracting cycles with rational ratio of periods. However, this non-coincidence originates rather from definition than from some interesting dynamical feature.

We are interested in attractors definitions of which rely on a reference measure (in our case, Lebesgue) on the phase space, which allows these attractors to capture asymptotic behavior of most points while possibly neglecting what happens with a set of orbits of zero measure. One type of such attractors was introduced by J. Milnor in~\cite{M} under the name ``the likely limit set''. We refer to it as the Milnor attractor.

\begin{defin}[Milnor attractor, \cite{M}] 
The Milnor attractor~$A_{Mil}(\varphi)$ of a dynamical system~$\varphi$ is the smallest closed set that contains the $\omega$-limit sets of $\mu$-almost all orbits.
\end{defin}

\emph{Statistical} and \emph{minimal} attractors were introduced in~\cite{AAIS} and~\cite{GI96}, respectively. We will use the definitions from~\cite{I05} and~\cite{K06} that describe the attractors in terms of so-called inessential sets.

For an open set $U \subset X$, let $I_U$ be the indicator function of~$U$ and let $I^t_U$ and $\tilde{I}^t_U$ be defined by the following equalities:
\[I^t_U(x) = \left((\varphi^t)^*I_U\right)(x) = I_U(\varphi^t(x)) = I_{(\varphi^t)^{-1}(U)}(x);\]
\[\tilde{I}^T_U(x) = \frac{1}{T}\int_0^T I^t_U(x) dt = \frac{1}{T}\int_0^T I_U(\varphi^t(x))dt.\]

\begin{defin}[Inessential sets]
An open set $U \subset X$ is \emph{statistically inessential} for a flow~$\varphi$ if $\tilde{I}^T_U \to 0$ as $T \to +\infty$ for $\mu$-a.e. $x \in X$.
An open set~$U$ is \emph{minimally inessential} if $\tilde{I}^T_U \to 0$ in measure, w.r.t. measure~$\mu$.
\end{defin}

\begin{rem}
Equivalently, an open set $U$ is minimally inessential if
\[\frac{1}{T}\int_0^T(\varphi^t_*\mu)(U)dt \to 0\; \text{ as } \; T \to +\infty.\]
\end{rem}

\begin{defin}[Statistical and minimal attractors, as defined in~\cite{I05} or~\cite{K06}]
The \emph{statistical (minimal) attractor}~$A_{stat}(\varphi)$ is the complement of the union of all statistically (resp., minimally) inessential open sets.
\end{defin}

\begin{rem}
The definition of Milnor attractor can be reformulated to resemble these two definitions. Namely, let us say that an open set $U \subset X$ is Milnor-inessential if $I^t_U(x) \to 0$ as $t \to +\infty$ for $\mu$-a.e. $x \in X$. Then, as it is not difficult to see, the Milnor attractor can be defined as the complement of the union of all Milnor-inessential open sets.
\end{rem}

\begin{rem}\label{rem:hierarchy}
Since, first, convergence of $I^t_U$ as $t \to +\infty$ implies convergence of $\tilde{I}^t_U$, and second, convergence $\mu$-a.e. implies convergence in measure (for finite measures), we always have the following hierarchy of attractors: $A_{min} \subset A_{stat} \subset A_{Mil}$.
\end{rem}

Another way to define an attractor is via an ``attracting'' invariant measure: the attractor is its support. Most suitable are the notions of \emph{physical} and \emph{natural} measures, which may be viewed as analogues of SRB-measures for general, non-hyperbolic dynamical systems (see, e.g.,~\cite{BB03}; we adapt the definitions from~\cite{BB03} to the case of flows). Physical measures describe the distribution of $\mu$-a.e. orbit, while natural measures capture the limit behavior of the reference measure.

For a flow $\varphi$ on a compact manifold~$X$ with measure~$\mu$, a probability measure $\nu$ is called \emph{physical} if there is a set~$B$ with $\mu(B) > 0$ such that for any~$x \in B$ and any continuous function~$f\in C(X, \RR)$ the Birkhoff time average over the orbit of~$x$ is equal to the space average w.r.t.~$\nu$:
\[\lim\limits_{T \to +\infty}\frac{1}{T}\int_0^T f(\varphi^t(x)) dt = \int_X f\, d\nu.\]
In other words, the measure~$\nu$ is physical if one has weak-$*$ convergence of the measures $\frac{1}{T}\int_0^T \delta_{\varphi^t(x)}dt$ to~$\nu$, for~$x \in B$.\footnote{Here $\delta_x$ stands for the~$\delta$-measure at~$x \in X$ and $\int_0^T \delta_{\varphi^t(x)}dt$ stands for a Borel measure~$\mu_{x,T}$ defined, due to the Riesz representation theorem, by the requirement that for any continuous function~$f\in C(X, \RR)$ one must have
\[\int_X f d\mu_{x, T} = \int_0^T(\delta_{\varphi^t(x)}, f) dt = \int_0^T f(\varphi^t(x)) dt.\]}
The set $B$ at which we have the convergence is called \emph{the basin of}~$\nu$. We will say that a physical measure is \emph{global} if its basin~$B$ has full measure~$\mu$.

A measure $\nu$ is called \emph{natural} for a flow~$\varphi$ if there exists an open subset $U \subset X$ such that for any probability measure $\tilde{\mu}$ absolutely continuous with respect to~$\mu$ and with $\mathrm{supp}(\tilde{\mu}) \subset U$ one has weak-$*$ convergence
\[\frac{1}{T}\int_0^T \varphi^t_*\tilde{\mu} \;dt \to \nu \; \text{ as } \; T \to +\infty.\]
That is, we apply the Krylov--Bogolyubov averaging procedure to~$\tilde\mu$ and, if the only
accumulation point for the averages is $\nu$, for any~$\tilde\mu$, we say that the measure~$\nu$ is natural. If we have this with $U = X$, we say that~$\nu$ is \emph{the global natural measure}.

By Theorem~2.1 of~\cite{BB03}, which can be straightforwardly adapted to the case of continuous time, for a dynamical system with a compact phase space a physical measure is always also natural.
We will sometimes refer to these measures as SRB-measures; however, we want to emphasize that we deal with physical and natural measures only, and not other variations that are referred to as SRB-measures in the literature.

Finally, for the square of a flow (or, more generally, for the product of two flows) we will define the attractors using the measure $\mu \times \mu$, or rather its completion (this does not change the attractors), which is a probability Lebesgue measure on $X \times X$.

%%%%%%%%%%%%%%%%%%%%%%%%%%%%%%%%%%%%%%%%%%%%%%%%%%%%%%%%%%%

\section{Milnor and statistical attractors for the square of the modified Bowen example}\label{sec:bowen}

The so-called \emph{Bowen example} is a vector field that has an attracting biangle formed by the separatrices of two hyperbolic saddles. In this example, under some mild genericity conditions, the time averages do not converge for Lebesgue-a.e. point that is attracted to the biangle~\cite[Theorem~1]{G92} (see also~\cite[Theorem~1]{T}), and, therefore, there is no physical measure. It was shown in~\cite{GK07} that, moreover, the time averages also do not converge for typical measures absolutely continuous w.r.t. the Lebesgue measure, and so there is no natural measure in this example.

\emph{The modified Bowen example} is obtained by replacing one of the saddles of the biangle by a contracting saddle-node. As V.~Kleptsyn showed in~\cite{K06}, for such a field the minimal attractor does not coincide with the statistical one: if we restrict the flow to a monodromic\footnote{That is, a semi-neighborhood where trajectories make winds near the polycycle. It exists because the saddle-node was chosen to be contracting.} semi-neighborhood of this polycycle, the minimal attractor will contain only the saddle-node while the statistical one will consist of both singularities.\footnote{Obviously, the Milnor attractor will coincide with the whole polycycle, since every trajectory in the semi-neighborhood is attracted to the polycycle.} We adopt the name of the example and also some notation from~\cite{K06}.

\begin{thm}\label{thm:stat}
There exists a $C^\infty$-smooth flow~$\varphi$ on~$S^2$ of codimension $3$, namely the modified Bowen example, such that for its Cartesian square~$\Phi = \varphi\times \varphi$ the Milnor attractor $A_{Mil}(\Phi)$ is a proper subset of the Cartesian square of~$A_{Mil}(\varphi)$, and the same holds for the statistical attractors:
$$A_{Mil}(\varphi\times \varphi) \subsetneqq A_{Mil}(\varphi) \times A_{Mil}(\varphi);$$
$$A_{stat}(\varphi\times \varphi) \subsetneqq A_{stat}(\varphi) \times A_{stat}(\varphi).$$
\end{thm}

\begin{figure}[ht]
\center{\includegraphics[width=0.7\linewidth]{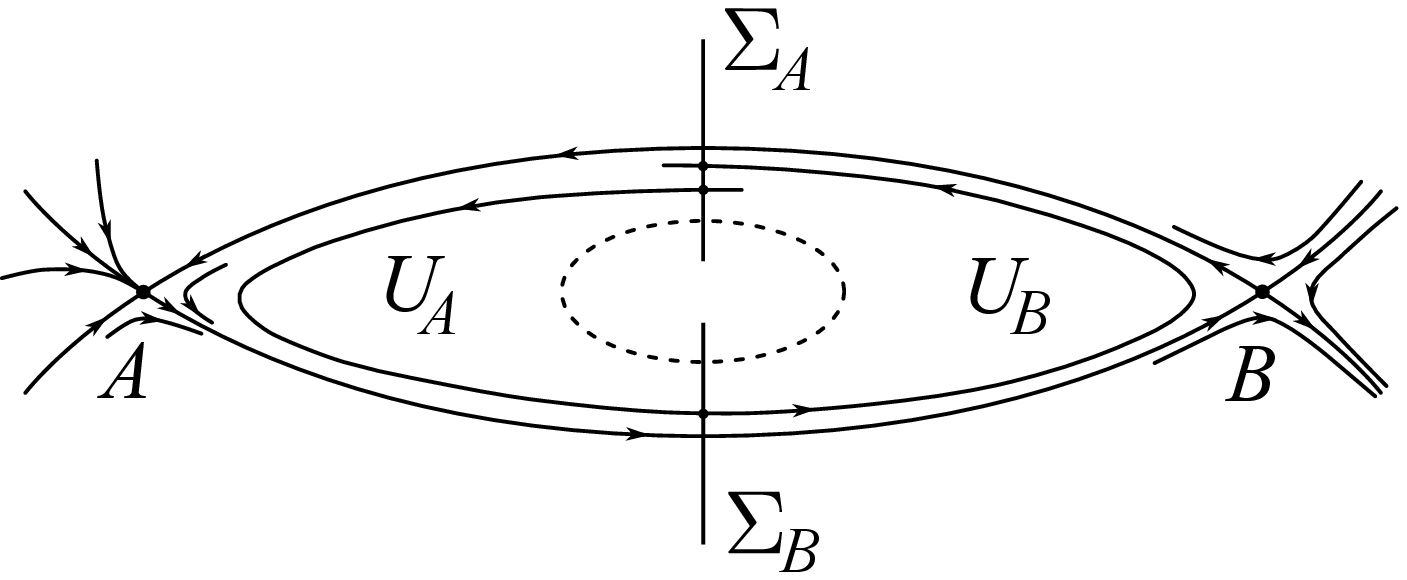}}
\caption{The phase portrait for the modified Bowen example.}\label{fig:1}
\end{figure}

\begin{rem}\label{rem:proj}
We always assume that the phase space~$X$ of~$\varphi$ is compact. The projection of $A_{Mil}(\Phi) \subset X \times X$ to~$X$ equals exactly~$A_{Mil}(\varphi)$. Indeed, a point $x \in X$ is not in this projection $\pi(A_{Mil}(\Phi))$ if and only if there is a neighborhood $U \ni x$ such that the open set~$U \times X$ is Milnor-inessential for~$\Phi$. But the latter condition on $U \times X$ is equivalent to~$U$ being Milnor-inessential for~$\varphi$. This means that~$x$  does not belong to $\pi(A_{Mil}(\Phi))$ iff it has an inessential neighborhood w.r.t~$\varphi$, which yields that $x \in \pi(A_{Mil}(\Phi))$ iff $x \in A_{Mil}(\varphi)$. The same argument is valid for statistical and minimal attractors, and it also works for arbitrary Cartesian products, not just squares. This implies that the attractor of the product is always a subset of the product of the attractors.
\end{rem}

\begin{proof}
Let our $\varphi$ be a flow that has a generic (that is, of order two\footnote{Here we refer to the order that the field has when restricted to the central manifold of the saddle-node. ``Order two'' means that the quadratic term of the restriction is non-zero.}) contracting saddle-node $A$ whose only outgoing separatrix coincides with an incoming separatrix of a saddle~$B$ and whose incoming separatrix (one of the two at the boundary of the parabolic sector of the saddle-node) coincides with an outgoing separatrix of~$B$. We assume that there are no further degeneracies, and so the flow is of codimension 3. The polycycle~$\Gamma$ formed by the singularities~$A$ and~$B$ and the two separatrices has a monodromic closed semi-neighborhood in which every trajectory either belongs to~$\Gamma$ or winds onto it. We may and will assume that our phase space~$X$ is not the whole $S^2$ but this semi-neighborhood, endowed with the normalized Lebesgue measure. Indeed, since our semi-neighborhood is forward-invariant for~$\varphi$ and has positive measure, if we prove the theorem with phase space being the semi-neighborhood, this will imply that it is also valid when $\varphi$ is viewed as a flow on~$S^2$.

We draw two transversals~$\Sigma_A, \Sigma_B$ to the separatrices that split $X$ into two domains~$U_A \ni A$ and~$U_B \ni B$, see Fig.~\ref{fig:1}. At some point below it will become important that we can choose the transversals arbitrarily, and so, for an appropriate choice of those, $U_A$ may be seen as a small neighborhood of the saddle-node~$A$.

As proven in Lemma~4 of~\cite{K06}, for any two points $x, y \in X\setminus\Gamma$ that do not share the same trajectory, there is a moment $t_0 > 0$ in time ($t_0$ depends on $x, y$) such that for $t > t_0$ the $\varphi^t$-images of these two points can never be in~$U_B$ together. We will deduce our theorem from this.

Let $\mu$ be the probability Lebesgue measure on~$X$. Since for a fixed $a \in X$ the set of $b \in X$ that belong to the orbit of $a$ has zero $\mu$-measure, the Fubini theorem yields, when considering the union of such sets over $a \in X$, that $(\mu \times \mu)$-a.e. pair $(a, b) \in X \times X$ has its points $a$ and~$b$ at different orbits. For a pair $(a, b)$ with this property, the aforementioned Lemma~4 of~\cite{K06} yields that the $\Phi$-orbit of $(a, b)$ never visits the set $U_B \times U_B$ when $t > t_0(a, b)$. This means that the open set $U_B \times U_B$ contains no points of the Milnor or statistical attractor. Thus, we conclude that $(B, B) \notin A_{Mil}(\varphi\times \varphi)$ and $(B, B) \notin A_{stat}(\varphi\times \varphi)$, while $(B, B) \in A_{stat}(\varphi) \times A_{stat}(\varphi) \subset A_{Mil}(\varphi) \times A_{Mil}(\varphi)$ since $B \in  A_{stat}(\varphi)$, which, when combined with Remark~\ref{rem:proj}, finishes the proof.
\end{proof}

For completeness we will also describe the attractors of $\Phi$ explicitly. Of course, the whole attractors depend on what happens outside of the semi-neighborhood~$X$, so in the following proposition we restrict ourselves to the phase space $X$, or $X \times X$ for the square flow.

\begin{prop}\label{prop:statstruct}
For the flow $\varphi$ as above the following holds:
$$A_{stat}(\varphi \times \varphi) = \{(A, A), (A, B), (B, A)\},$$
$$A_{Mil}(\varphi \times \varphi) = \left(\Gamma \times \{A\}\right) \cup \left(\{A\} \times \Gamma\right).$$
\end{prop}
\begin{proof}
The statistical attractor of $\Phi$ is easy to find. Both natural projections ${X \times X \to X}$ must take it to $A_{stat}(\varphi)$ = $\{A, \, B\}$. Also, since ${A_{min}(\varphi) = \{A\}}$, we have that $\{(A, A)\} = A_{min}(\Phi) \subset A_{stat}(\Phi)$. As we know, the point $(B, B)$ is not in the attractor. Hence, there is only one possibility, namely: $A_{stat}(\Phi) = \{(A, A), (A, B), (B, A)\}$.

Now let us consider the Milnor attractor. Denote the set $\left(\Gamma \times \{A\}\right) \cup \left(\{A\} \times \Gamma\right)$ by $\Om$ for brevity.
Recall that by taking the transversals to the separatrices close to the saddle-node~$A$ we make $U_A$ be a small neighborhood of $A$ (in the space $X$). Thus the domain $U_B$ contains almost the whole polycycle~$\Gamma$. But, as we proved above, $U_B \times U_B$ contains no points of the Milnor attractor, so the attractor must be contained in the set
$$ (X \times X) \setminus (U_B \times U_B) = \overline{(X \times U_A) \cup (U_A \times X)}.$$
Keeping in mind that $A_{Mil}(\Phi)$ must be a subset of $A_{max}(\Phi, X^2) = \Gamma \times \Gamma$ and that we can make $U_A$ as small as we want, we conclude that $A_{Mil}(\Phi)$ must be a subset of~$\Om$. On the other hand, the natural projections of $A_{Mil}(\Phi)$ are both $A_{Mil}(\varphi)$. The only proper subset of~$\Om$ with this property is $\Om \setminus \{(A, A)\}$, but the Milnor attractor must be closed, hence $A_{Mil}(\Phi) = \Om$.
\end{proof}

\section{Milnor attractor for the square of an attracting separatrix loop}\label{sec:loop}

N.~Agarwal, A.~Rodrigues, and M.~Field proved (see~\cite[Theorem~1.1]{ARF}) that for the product of two planar vector fields with so-called heteroclinic attractors (attracting polycycles formed by hyperbolic saddles and their separatrices) the Milnor attractor (for appropriately restricted phase space) is a proper subset of the product of these attractors; this subset is a union of two products where the first heteroclinic attractor is multiplied by the discrete set of saddles of the second one, and vice versa. A special case where, instead of a product of arbitrary heteroclinic attractors, the square of one saddle loop is considered admits a much shorter proof that we would like to present in this section.

Let $\varphi$ be a smooth flow on $S^2$ with an attracting separatrix loop~$\gamma$ of a generic saddle\footnote{By `generic' we mean that the characteristic number (minus ratio of eigenvalues; the negative one goes to the numerator) is not equal to $1$, and for the loop to be attracting, as it is well-known, we need the characteristic number to be~$\ge 1$.}~$A$. Arguably this is the simplest flow for which we can observe the phenomenon of non-coincidence of the square of the attractor and the attractor of the square of the flow. However, this non-coincidence happens only for Milnor attractors. Indeed, if we restrict~$\varphi$ to a small monodromic semi-neighborhood of $\overline\gamma = \gamma \cup \{A\}$,  the Milnor attractor will be equal to~$\overline\gamma$, since it is the $\omega$-limit set for any orbit that winds onto the loop, while $A_{stat}(\varphi)$ (and hence also $A_{min}(\varphi)$) will contain only one point, the saddle: as a forward semi-orbit approaches the loop, the fraction of time it spends near the saddle tends to one\footnote{The same argument applies to a loop whose saddle has characteristic number~1 if the loop happens to be attracting.}. This leaves no room for non-coincidence in case of statistical or minimal attractors.

Let~$\Sigma_A$ be a transversal to the incoming local separatrix involved into the loop and let~$\Sigma_B$ be a transversal to the outgoing local separatrix. We will assume that~$\Sigma_A, \Sigma_B$ are taken close to the saddle~$A$ and that they split a small monodromic closed semi-neighborhood~$X$ of~$\overline\gamma$ into two domains $U_A \ni A$ and~$U_B$. There is no singularity~$B$ this time, but we emphasize the analogy with Theorem~\ref{thm:stat} by using the same notation.

\begin{thm}\label{thm:mil}
There exists a $C^\infty$-smooth flow~$\varphi$ of codimension $1$, namely, a flow with a generic attracting separatrix loop, such that for its Cartesian square~$\Phi = \varphi\times \varphi$ we have
$$A_{Mil}(\Phi) \subsetneqq A_{Mil}(\varphi) \times A_{Mil}(\varphi).$$
\end{thm}

This theorem is proved like Theorem~\ref{thm:stat}, but with reference to Lemma~4 of~\cite{K06} replaced by Proposition~\ref{prop:loop} below. If we restrict $\varphi$ to $X$, we will be able to describe the attractor of $\Phi$ explicitly:
\[A_{Mil}(\Phi) = \left(\overline\gamma \times \{A\}\right) \cup \left(\{A\} \times \overline\gamma\right);\]
the proof is as in Proposition~\ref{prop:statstruct}.

\begin{prop}\label{prop:loop}
For a smooth flow $\varphi$ with a generic attracting saddle loop as above, for any two points~$x, y \in X$ that do not belong to the same trajectory, there is a point in time $t_0 > 0$ such that for $t > t_0$ the $\varphi^t$-images of $x, y$ cannot be in the set~$U_B$ together.
\end{prop}

\begin{proof}
For any smooth coordinates on~$\Sigma_A, \, \Sigma_B$ that have zeros on the separatrices, the monodromy map\footnote{The subscript $s$ in $\Delta_s$ stands for `saddle'. Below we will also consider a similar map $\Delta_{sn}$ for a saddle-node.}~$\Delta_s$ from $\Sigma_A$ to $\Sigma_B$ has the form $\xi \mapsto c\cdot \xi^\nu(1 + o(1))$ as $\xi \to 0$, where $\nu$ is the characteristic number of the saddle, and the time it takes to pass from $\Sigma_A$ to $\Sigma_B$ is $\tau(\xi) = {- c_1\cdot\log\xi (1 + o(1))}$, see Lemma~13.1 of~\cite{SSTC} and Lemma~1 of~\cite{K06}. The freedom in choosing the coordinates on the transversals allows us to arrange that the monodromy from $\Sigma_B$ to $\Sigma_A$ be $x \mapsto x$ and so the total Poincar\'e map $P$ coincide, in coordinate representation, with the monodromy map~$\Delta_s$.

Let us denote by $\Delta$ the Poincar\'e map~$P$ rewritten in a new coordinate~$\zeta = -\ln\xi$.

\begin{lem}
The map $\Delta$ has two important properties:

i) it is expanding for large~$\zeta$; 

ii) given any two points~$\zeta > \hat{\zeta} > 0$, the ratio of their iterates stays bounded away from 1.
\end{lem}

\begin{proof}

We start with proving that the map is expanding. Since
\[\Delta(\zeta) = -\ln(P(e^{-\zeta})),\]
for its derivative we have
\[\Delta'(\zeta) = \frac{P'(e^{-\zeta}) \cdot e^{-\zeta}}{P(e^{-\zeta})}.\]
For a monodromy map~$\Delta_s$, one has $\lim_{x \to 0} \frac{x\Delta'_s(x)}{\Delta_s(x)} = \nu$, where $\nu$ is the characteristic number.\footnote{Indeed, there exist $C^1$-smooth linearizing coordinates in a neighborhood of the saddle (see~\cite{N}) that induce charts on the transversals in which the monodromy map is just $x \mapsto x^\nu$, and the claim holds. Then it is straightforward to check that a $C^1$ change of coordinate (that preserves the origin) on any of the two transversals does not change the limit.} This yields that $\frac{e^{-\zeta} P'(e^{-\zeta})}{P(e^{-\zeta})} = \nu+o(1)$ as $\zeta \to +\infty$, and hence
\[\Delta'(\zeta) = \nu+o(1).\]
We assumed that the loop is attracting and the saddle is generic, so~$\nu > 1$, and hence the map~$\Delta$ is expanding for large~$\zeta$.

Now let us establish the second property. The asymptotic formula for~$P$ provides the following asymptotic formula for $\Delta$:
\[\Delta(\zeta) = \nu\zeta - \ln c + o(1), \text{ as } \zeta \to +\infty.\]
Let $\zeta$ be greater than $\hat{\zeta}$ and let both be so large that $o(1)$ in the formula above is smaller than $1$.
Then we set $C = |\ln c| + 1$ and write
\[\frac{\Delta(\zeta)}{\Delta(\hat{\zeta})} = \frac{\nu\zeta - \ln c + o(1)}{\nu\hat{\zeta} - \ln c + o(1)} > \frac{\nu\zeta - \nu(|\ln c| + 1)}{\nu\hat{\zeta} + \nu(|\ln c| + 1)} = \frac{\zeta - C}{\hat{\zeta} + C}.\]

First assume that we have $\frac{\zeta}{\hat{\zeta}} < 2$. Then, as it is easy to see, $\frac{\zeta + \hat{\zeta}}{\hat\zeta(\hat\zeta + C)} < 3/\hat\zeta$ and we can write
\[\frac{\Delta(\zeta)}{\Delta(\hat{\zeta})}  > \frac{\zeta - C}{\hat{\zeta} + C} = \frac{\zeta}{\hat{\zeta}} - \frac{C(\zeta + \hat{\zeta})}{\hat\zeta(\hat\zeta + C)} > \frac{\zeta}{\hat{\zeta}} - 3C/\hat\zeta.\]

This inductively yields that 
\[\frac{\Delta^n(\zeta)}{\Delta^n(\hat{\zeta})} > \frac{\zeta}{\hat{\zeta}} - 3C \sum_{k=0}^{n-1}\frac{1}{\Delta^{k}(\hat\zeta)}.\]

The asymptotic formula above shows that $\frac{1}{\Delta^{k}(\hat\zeta)}$ decreases exponentially as $k \to +\infty$, so we have, for any $n$,
\[\sum_{k=0}^{n-1}\frac{1}{\Delta^{k}(\hat\zeta)} < \sum_{k=0}^{+\infty}\frac{1}{\Delta^{k}(\hat\zeta)} < \frac{const}{\hat{\zeta}},\]
and thus
\[\frac{\Delta^n(\zeta)}{\Delta^n(\hat{\zeta})} > \frac{\zeta}{\hat{\zeta}} -\frac{const}{\hat\zeta} = 1 + \frac{(\zeta - \hat\zeta) - const}{\hat\zeta}.\]
Since the map $\Delta$ is expanding, we can replace $\zeta, \hat{\zeta}$ by their iterates, if necessary, and assume that the numerator $(\zeta - \hat\zeta) - const$ is positive and we have the required estimate.

Now consider the case $\frac{\zeta}{\hat{\zeta}} \geq 2$. Here we can write
\[\frac{\Delta(\zeta)}{\Delta(\hat{\zeta})}  > \frac{\zeta - C}{\hat{\zeta} + C}  = \frac{\zeta}{\hat{\zeta}} \cdot \frac{1 - C/\zeta}{1 + C/\hat{\zeta}} \geq \frac{\zeta}{\hat{\zeta}} \cdot \left(1 - 2C/\hat{\zeta} \right),\]

and, inductively,
\[\frac{\Delta^n(\zeta)}{\Delta^n(\hat{\zeta})} > \frac{\zeta}{\hat{\zeta}} \cdot \prod_{k = 0}^{n-1} (1 - 2C/\Delta^k(\hat{\zeta})) = \frac{\zeta}{\hat{\zeta}} \cdot \exp\left(\sum_{k=0}^{n-1} \ln(1 - 2C/\Delta^k(\hat{\zeta}))\right).\]
For small positive $x$ we have $\ln(1 - x) > -2x$, so the sum of logarithms is greater than $-4C \cdot \sum_{k=0}^{+\infty}\frac{1}{\Delta^{k}(\hat\zeta)} > -\frac{const}{\hat{\zeta}}$, and so the exponent is greater than $\exp(-\frac{const}{\hat{\zeta}})$, and the latter is close to one when $\hat{\zeta}$ is sufficiently large. Thus, since we assumed that~$\frac{\zeta}{\hat{\zeta}} \geq 2$, we have that the ratios $\frac{\Delta^n(\zeta)}{\Delta^n(\hat{\zeta})}$ are separated from one, and we are done. 
\end{proof}

Consider the positive semi-trajectories of $x, y$ and let $x_1, y_1 \in \Sigma_A$ be two points of these semi-trajectories such that the point~$y_1$ is in the segment of~$\Sigma_A$ between~$x_1$ and its Poincar{\'e}-map preimage $P^{-1}(x_1)$.

Let $t_{n, x_1}$ and $t_{n, y_1}$ be the moments of time $t$ when the $\varphi^t$-images of~$x_1$ and~$y_1$ cross the transversal~$\Sigma_B$ for the~$n$-th time. Our goal is to show that $t_{n, x_1} - t_{n, y_1} \to +\infty$ and $t_{n+1, y_1} - t_{n, x_1} \to +\infty$ as $n \to +\infty$.

Given $z \in \Sigma_A$, denote the $\zeta$-coordinate of $P^n(z)$ by~$\zeta_n(z)$, and given $\zeta$, denote by~$t(\zeta)$ the time it takes for a point of~$\Sigma_A$ with coordinate~$\zeta$ to get to~$\Sigma_B$, i.e., the time associated with the monodromy~$\Delta_s$, discussed above: $t(\zeta) = c_1\zeta \cdot (1 + o(1))$ as $\zeta \to +\infty$. The time of one rotation around the loop has the same asymptotic. The ratios $\zeta_n(x_1)/\zeta_n(y_1)$ and $\zeta_n(y_1)/\zeta_{n-1}(x_1)$ are bounded away from 1, by the second property of~$\Delta$. Therefore, we may assume, replacing $x_1, y_1$ by their images under the iterates of~$P$, if necessary, that the point~$x_1$ makes winds slower than~$y_1$, but faster than $P(y_1)$, and so the sequences $t_{n, x_1}, t_{n, y_1}$ are ordered in the following way:
$$t_{1, y_1} < t_{1, x_1} < t_{2, y_1} < t_{2, x_1} < \dots < t_{n, y_1} < t_{n, x_1} < t_{n+1, y_1}< \dots$$

This assumption also means that, when the $\varphi^t$-image of $y_1$ has made $n-1$ rotations around the loop and returned to~$\Sigma_A$, the image of $x_1$ is not yet at $\Sigma_A$, so we have, when $n\to +\infty$,
\[t_{n, x_1} - t_{n, y_1} >  t(\zeta_{n-1}(x_1)) - t(\zeta_{n-1}(y_1)).\]

Let us consider the ratio of times from the right hand side. We have
\[\frac{t(\zeta_{n-1}(x_1))}{t(\zeta_{n-1}(y_1))} = \frac{\zeta_{n-1}(x_1)}{\zeta_{n-1}(y_1)} \cdot (1 +o(1)) \text{ as } n \to +\infty.\]
Since the ratio~$\frac{\zeta_{n-1}(x_1)}{\zeta_{n-1}(y_1)}$ is bounded away from 1, the same holds for the ratio of times when $n$ is large, but then the difference of these times also tends to infinity as $n \to +\infty$. Hence, we have 
\[t_{n, x_1} - t_{n, y_1} >  t(\zeta_{n-1}(x_1)) - t(\zeta_{n-1}(y_1)) \to +\infty \text{ as } n \to +\infty.\]
Analogously, when the image of $x_1$ arrives at $\Sigma_A$ after $n-1$ rotations, the image of~$y_1$ has not finished its $n$-th rotation yet, so we have
\[t_{n+1, y_1} - t_{n, x_1} >  t(\zeta_{n}(y_1)) - t(\zeta_{n-1}(x_1)) \to +\infty.\]

The time needed to cross the domain~$U_B$ is bounded from above by some constant~$K$ independent of the orbit. Since $t_{n, y_1} < t_{n, x_1} < t_{n+1, y_1}$, the images of $x_1, y_1$ can be in~$U_B$ simultaneously while $x_1$ makes its $n$-th wind only if $t_{n, x_1} - t_{n, y_1} < K$ or if $t_{n+1, y_1} - t_{n, x_1} < K$, which can happen only for finitely many~$n$ because of the divergence to infinity established above. Hence, there is a moment of time after which the $\varphi^t$-images of the points~$x_1, y_1$ never visit~$U_B$ together. Since the sequences $t_{n, x}$ and $t_{n, y}$ defined analogously for the original two points $x, y$ differ from $t_{n, x_1}$ and $t_{n, y_1}$ by two fixed constants, the same conclusion holds for~$x, y$.
This proves the proposition.
\end{proof}

%%%%%%%%%%%%%%%%%%%%%%%%%%%%%%%%%%%%%%%%%%%%%%%%%%%%%%%%%%%%
\section{Statistical attractors for products of flows of codimension at most~2}\label{sec:lune}

In this section we prove that an example of non-coincidence for statistical attractors like the one from Section~\ref{sec:bowen} cannot be obtained by taking a product of two flows of codimension at most two, at least on the sphere.

\begin{thm}\label{thm:codim2}
For any two smooth flows on $S^2$ of codimension at most 2 (i.e., flows that may appear in generic two-parameter families), the statistical attractor of their product coincides with the product of their statistical attractors.
\end{thm}

\subsubsection*{There are finitely many singularities, cycles, and polycycles}
The topological classification of singularities that appear in vector fields of codimension at most 2 is known; those are (not necessarily hyperbolic) sinks, sources, saddles, contracting and dispersing saddle-nodes of order two and cusps (see, e.g.,~\cite{KS} and references therein). Singularities of these types are isolated and the sphere is compact, so there are finitely many singularities. Also, such singularities obviously have neighborhoods that contain no cycles. Hence, cycles, if there were infinitely many of them, could only accumulate to a cycle or a polycycle. The first possibility is ruled out by the observation that degenerate cycles in fields of codimension at most two can only have eigenvalue one and, maybe, another single degeneracy in the nonlinear part of the Poincar\'e map (see~\cite[Ch.~2, S.~1-2]{AAIS}). The second is ruled out by the results of~\cite{T96} where it was shown that, for any polycycle that can appear in a generic 3-parameter family of vector fields, at most 3 cycles may appear near this polycycle, even after a small perturbation of the field.

\subsubsection*{Essential limit points}

By the Poincar\'e--Bendixson theorem, for the vector field in the sphere with finitely many singularities the $\omega$-limit set of a point is either a cycle, a singularity, or a polycycle.
Since in our case there are only finitely many of those, the union of all $\omega$-limit sets is closed, and, therefore, the statistical attractor can only include trajectories that are in some $\omega$-limit sets. Obviously, sinks and contracting saddle-nodes belong to the statistical attractor and sources do not. A polycycle can be the $\omega$-limit set of a point only if this polycycle is monodromic and the orbit of this point winds onto it. As the orbit does so, the fraction of time that the orbit spends near the union of the polycycle's singularities tends to one, so only singularities of the polycycle can belong to~$A_{stat}$.

A cycle of our vector field is isolated, so it can be either repelling, or attracting, or semi-stable. In the latter two cases it is a part of~$A_{stat}$: it attracts a set of points that has positive measure, so it must contain a point of the attractor, but the attractor consists of whole trajectories. The upshot is that if a cycle is in the attractor, it has at least one semi-neighborhood where it attracts every point. 

\subsubsection*{The products of elements are in the attractor}

Statistical attractors of flows of codimension at most 2 consist of elements of the following three types: sinks or contracting saddle-nodes\footnote{In particular, if a vector field has a contracting saddle-node with a homoclinic curve, then the saddle-node belongs to the statistical attractor, but the points of the homoclinic curve do not, and our argument will not distinguish this case from the case where there is no homoclinic curve.}, cycles, sets of singularities of monodromic polycycles. Though the first and third types intersect, it should not concern us. In order to prove Theorem~\ref{thm:codim2}, it is sufficient to prove that the product of elements of these types is a part of $A_{stat}$ for the product of flows.

Let $E_1$ be such an element for the flow~$\varphi$ and $E_2$ be such an element for~$\tilde{\varphi}$. It is not difficult to see that if any of the elements~$E_1, E_2$ is a one-point set, the product~$E_1 \times E_2$ is a subset of $A_{stat}(\varphi \times \tilde{\varphi})$.

The first case when both~$E_1$ and~$E_2$ have more than one point is when they are both cycles.
Let $E_1$ and $E_2$ be the cycles of~$\varphi$ and~$\tilde{\varphi}$ and $U_1, \, U_2$ be their semi-neighborhoods where the points are attracted to the cycles.
The torus $\TT = E_1 \times E_2$ attracts the set $U_1 \times U_2$ of positive product measure and therefore contains a point of the statistical attractor. Let this point be~$x = (a,b)$ and let~$U$ be its arbitrary neighborhood. There is a set~$S$ of positive product measure that indicates that $U$ is not statistically inessential: this set consists of points $z$ such that
\[\limsup_{T \to +\infty}\frac{1}{T}\int_0^T I_U((\varphi\times\tilde\varphi)^t(z))dt > 0.\]
If such a set had zero measure, $U$ would be inessential. 

Fix an arbitrary~$t \in \RR$ and define $\hat S$ as the image of~$S$ under the map $(u, v) \mapsto (\varphi^t(u), v)$. This map is a diffeomorphism, and therefore, since~$S$ has positive measure, $\hat S$ is also of positive measure. But~$\hat S$ indicates that an open set $\hat{U} \ni (\varphi^t (a), b)$ is not inessential. Since $U$ was arbitrary, we argue that any neighborhood~$\hat{U}$ of $(\varphi^t (a), b)$ is not inessential, which yields that $(\varphi^t (a), b)$ is in~$A_{stat}$. Generalizing this argument, we get that $(\varphi^{t_1} (a), \tilde{\varphi}^{t_2}(b))\in A_{stat}(\varphi \times \tilde{\varphi})$ for arbitrary~$ t_1, t_2$ and hence the whole torus $E_1 \times E_2$ is in the attractor.

The second case is when $E_1$ is a cycle and~$E_2$ is a set of singularities of a monodromic polycycle with more than one singularity, or vice versa. There is only one such polycycle in the list (see \cite{KS} and~\cite{T96}) for codimension~$\le 2$, namely, the attracting biangle: two hyperbolic singularities with two separatrix connections arranged so that there is a monodromic semi-neighborhood where points are attracted to the polycycle.

The key idea here is that for the point in the semi-neighborhood of the cycle~$E_1$ each rotation around~$E_1$ takes bounded time asymptotically equal to the period of the cycle, while for the biangle~$E_2$ the times of rotations diverge to infinity. The product $E_1 \times E_2$ is a union of two circles, so one can modify the argument from the previous case to first show that each circle has a point of the attractor and then argue that the intersection of $E_1 \times E_2$ with the attractor is invariant under the maps of the form $(u, v) \mapsto (\varphi^t(u), v)$.

Finally, the third case is the product of (the sets of singularities of) two attracting biangles, and it is the only nontrivial case. We solve it in the following proposition.

\begin{prop}\label{prop:biangle}
Consider two smooth flows $\varphi$ and $\tilde{\varphi}$ on~$S^2$, each with an attracting biangle, that is, a monodromic polycycle with attracting semi-neighborhood formed by two separatrix connections between two hyperbolic saddles. Denote the biangle of~$\varphi$ by~$\Gamma$, its two hyperbolic saddles by~$A, B$, and denote a monodromic semi-neighborhood of the polycycle where all orbits are attracted to it by~$X$; use the same notation with tilde for~$\tilde{\varphi}$. Suppose that for both polycycles the product of characteristic numbers of the two saddles is greater than~$1$. Then, if~$\varphi$ and~$\tilde{\varphi}$ are considered as dynamical systems with phase spaces~$X$ and $\tilde{X}$ and we set~$\Phi = \varphi \times \tilde{\varphi}$, the statistical attractor of~$\Phi$ consists of four points:
\[A_{stat}(\Phi) = \{(A, \tilde{A}), (A, \tilde{B}), (B, \tilde{A}), (B, \tilde{B})\}.\]
\end{prop}

\begin{proof}

Let the monodromic semi-neighborhood\footnote{We consider the closed semi-neighborhood that includes the polycycle and, in particular, the two saddles.}~$X$ of the biangle be split into domains $U_A \ni A, \; U_B \ni B$ by the transversals $\Sigma_A, \Sigma_B$ to the two separatrix connections. The transversal~$\Sigma_A$ is the one through which trajectories in~$X$ enter~$U_A$. We use the same notation with tilde for the analogous objects of~$\tilde{\varphi}$.

The statistical attractor of~$\varphi$ is a two-point set~$\{A, B\}$: the main result of~\cite{GK07} implies that the minimal attractor coincides with this set, and the statistical attractor cannot be smaller. Since
\[A_{stat}(\Phi) \subset A_{stat}(\varphi) \times A_{stat}(\tilde{\varphi}) = \{A, B\} \times \{\tilde{A}, \tilde{B}\},\]
it suffices to prove that each of the four domains $U_A \times U_{\tilde{A}}, \; U_A \times U_{\tilde{B}}, \; U_B \times U_{\tilde{A}}, \; U_B \times U_{\tilde{B}}$ is not statistically inessential for the flow~$\Phi$.

Consider the flow $\varphi$ first.
Let the eigenvalues of $A$ and $B$ be $-\mu_A, \lambda_A$ and $-\mu_B, \lambda_B$ respectively. For an arbitrary point $z \in \Sigma_A\setminus\Gamma$, denote by $T_{k, A}(z)$ the moment of time when the positive semi-orbit ${\mathcal O}_+(z)$ returns to~$\Sigma_A$ for the $k$-th time and by~$T_{k, B}(z)$ the moment of time when this semi-orbit intersects~$\Sigma_B$ for the first time after~$T_{k, A}(z)$. Let $\Lambda = \frac{\mu_A \mu_B}{\lambda_A \lambda_B} > 1$ be the product of the characteristic numbers of the saddles~$A, B$ and let $\Lambda_0 = \frac{\Lambda + \mu_A/\lambda_B}{1 + \mu_A/\lambda_B}$.
By Lemma~2 of~\cite{GK07}, there is a continuous positive function~$\gamma$ defined on the open semi-transversal~$\Sigma_A \setminus \Gamma$ such that
\begin{equation}\label{eq:geometric}
T_{k, A}(z) = \Lambda^k \cdot (\gamma(z) + o(1)) \;\; \text{ as } k \to +\infty.
\end{equation}
The meaning of this is that the moments $T_{k, A}(z)$ almost form a geometric series. 
Note that~\eqref{eq:geometric} implies $\gamma(P(z)) = \Lambda\gamma(z)$, where $P$ is the Poincar\'e map associated with the polycycle, and so $\gamma(z) \to +\infty$ as $z \to 0$ along the transversal, where $\{0\} = \Sigma_A \cap \Gamma$.
By Proposition~3 of~\cite{GK07}, we also have
\begin{equation}\label{eq:geometric2}
T_{k, B}(z) = \Lambda_0\Lambda^k \cdot (\gamma(z) + o(1)) \;\;  \text{ as } k \to +\infty.
\end{equation}

Let us introduce the logarithmic time $\tau = \tau(t) = \log_\Lambda(t)$. Using notation $\tau_{k, A}(z) = \tau(T_{k, A}(z))$ and $\tau_{k, B}(z) = \tau(T_{k, B}(z))$ we can rewrite~\eqref{eq:geometric} and~\eqref{eq:geometric2} as
\begin{equation}\label{eq:sec1}
\tau_{k, A}(z) = \log_\Lambda\gamma(z) + k + o(1),
\end{equation}
\begin{equation}\label{eq:sec2}
\tau_{k, B}(z) = \log_\Lambda\Lambda_0 + \log_\Lambda\gamma(z) + k + o(1) \text{ as } k \to +\infty.
\end{equation}

The same applies to the biangle of~$\tilde{\varphi}$, so for a given $\tilde{z} \in \Sigma_{\tilde A}\setminus\tilde{\Gamma}$  there are sequences $T_{k, \tilde{A}}(\tilde{z})$ and~$T_{k, \tilde{B}}(\tilde{z})$ which we want to rewrite in the coordinate~$\tau$ introduced above. We use exactly the same $\tau$, not its analogue for the second flow. Using the notation $\tau_{k, \tilde{A}}(\tilde{z}) = \tau(T_{k, \tilde{A}}(\tilde{z}))$ and $\tau_{k, \tilde{B}}(\tilde{z}) = \tau(T_{k, \tilde{B}}(\tilde{z}))$, we naturally have
\begin{equation}\label{eq:sec3}
\tau_{k, \tilde{A}}(\tilde{z}) = \log_\Lambda\tilde{\gamma}(\tilde{z}) + k \cdot \log_\Lambda\tilde{\Lambda} + o(1),
\end{equation}
\begin{equation}\label{eq:sec4}
\tau_{k, \tilde{B}}(\tilde{z}) = \log_{\Lambda}\tilde{\Lambda}_0 +
\log_\Lambda\tilde{\gamma}(\tilde{z}) + k \cdot \log_\Lambda\tilde{\Lambda} + o(1) \text{ as } k \to +\infty.
\end{equation}

The sequences~\eqref{eq:sec1},\eqref{eq:sec2} and~\eqref{eq:sec3},\eqref{eq:sec4} provide two partitions of the $\tau$-axis. Let us assign colors to the intervals of these two partitions: let the intervals $(\tau_{k, A}(z), \tau_{k, B}(z))$ be white,
the intervals $(\tau_{k, B}(z), \tau_{k+1, A}(z))$ be black,
the intervals $(\tau_{k, \tilde{A}}(\tilde{z}), \tau_{k, \tilde{B}}(\tilde{z}))$ be blue, and
the intervals $(\tau_{k, \tilde{B}}(\tilde{z}), \tau_{k + 1, \tilde{A}}(\tilde{z}))$ be red. Thus, white intervals correspond to $\varphi^{t(\tau)}(z)$ being in~$U_A$, black ones correspond to it being in~$U_B$, blue ones correspond to the $\tilde{\varphi}^t$-image of~$\tilde{z}$ being in~$U_{\tilde{A}}$, and the red ones are those where it is in~$U_{\tilde{B}}$.

It is sufficient to prove that for each of the four combinations of colors --- white-blue, white-red, black-blue, black-red --- there is a set of pairs~$(z, \tilde{z})$ of positive measure\footnote{Since $z$ and~$\tilde{z}$ are taken from the transversals~$\Sigma_A$ and~$\Sigma_{\tilde{A}}$ in our construction, here we mean the product of the one-dimensional Lebesgue measures on these transversals.} such that for each pair~$(z, \tilde{z})$ there are infinitely many intersections, with length bounded away from zero, of intervals of the two colors. Indeed, suppose that there is a segment~$[\tau_1, \tau_2]$ of length $\e > 0$ that is in the intersection of a black and a red intervals. Since $\tau$ is the logarithmic time, at $t$-time equal to $\Lambda^{\tau_2}$ the fraction of time spent by the $\Phi$-orbit of~$(z, \tilde{z})$ in the set~$U_{B\tilde{B}} = U_B \times U_{\tilde{B}}$ is at least
\[\frac{\Lambda^{\tau_2} - \Lambda^{\tau_1}}{\Lambda^{\tau_2}} = 1 - \Lambda^{\tau_1 -\tau_2} = 1 - \Lambda^{-\e} > 0,\]
and so, if there are infinitely many such intervals, this fraction does not converge to zero for this orbit.
The same argument applies if we replace~$z$ by another point~$x$ of its~$\varphi$-orbit, or $\tilde{z}$ by another point~$\tilde{x}$ of its~$\tilde{\varphi}$-orbit, because the four sequences analogous to~\eqref{eq:sec1}--\eqref{eq:sec4} will be the same for the new points, up to residual terms of class~$o(1)$. But since the pairs~$(z, \tilde{z})$ form a positive measure set in the product of the transversals, the pairs $(x, \tilde{x})$ will form a set of positive product measure in the total phase space~$X \times \tilde{X}$. Then the set~$U_{B\tilde{B}}$ cannot be statistically inessential for~$\Phi$, which implies that its closure contains a point of the attractor, but this point can only be~$(B, \tilde{B})$. The argument for the other three pairs of colors is the same.
 
Suppose first that $\log_\Lambda\tilde{\Lambda} = \ln\tilde{\Lambda}/\ln\Lambda$ is irrational. The sequences~\eqref{eq:sec1}--\eqref{eq:sec4} are almost arithmetic progressions, but with common differences that have irrational ratio. If there were no residual terms of class $o(1)$ in these sequences and we decided to consider the factor of the $\tau$-axis by a unit shift, the two sequences~\eqref{eq:sec1} and~\eqref{eq:sec2} would factor into two points at the factor-circle and the sequences~\eqref{eq:sec3} and~\eqref{eq:sec4} would become the orbits of an irrational circle rotation by $\log_\Lambda\tilde{\Lambda}$.

Any forward semi-orbit of an irrational rotation is everywhere dense. Hence, for any interval~$I$ of the $\tau$-axis, any orbit of the irrational shift $\tau \mapsto \tau + \log_\Lambda\tilde{\Lambda}$ visits the set~$I + \ZZ = \{x + n \mid x \in I, \; n \in \ZZ\}$ infinitely many times. Then so does sequence~\eqref{eq:sec3}, because it is asymptotically an orbit of an irrational shift.

Let $I$ be a small neighborhood of the point~$\log_\Lambda(\gamma(z))$. Then, if an element of the sequence~\eqref{eq:sec3} with large index is in~$I + \ZZ$, we have almost coincidence of the left edges of a white and a blue interval, and hence a huge, with respect to the real time~$t$, intersection of these intervals. The argument for the other color pairs is analogous. Thus we get infinitely many huge intersections of intervals for any pair of colors and for any pair~$(z, \tilde{z})$, and we are done with the irrational case.

Now suppose that $\log_\Lambda\tilde{\Lambda} = \frac{p}{q} \in \QQ$. Let us consider the white-blue intersections; the argument for other color pairs is analogous. Fix any $z \in \Sigma_A \setminus \Gamma$ and choose~$\tilde{z}$ such that
\[\log_\Lambda\tilde{\gamma}(\tilde{z}) - \log_\Lambda\gamma({z}) \in \ZZ.\]
It is indeed possible to find such $\tilde{z}$ because $\tilde{\gamma}(\tilde{z})$ tends to $+\infty$ as $\tilde{z} \to 0$.
Then for~$(z, \tilde{z})$ we have asymptotic coincidence of the left edges of blue intervals with numbers~$qn,\; n\in\nn,$ and the left edges of white intervals with numbers~$a + pn$ (here~$a = \log_\Lambda\tilde{\gamma}(\tilde{z}) - \log_\Lambda\gamma({z})$). If we slightly alter $\tilde{z}$, because of continuity of~$\tilde{\gamma}$ we will still get huge intersections of the intervals. So, for a fixed $z$, there is at least a segment of $\tilde{z}$-s which produce the white-blue intersections. By the Fubini theorem, the set of pairs~$(z, \tilde{z})$ for which we have infinitely many huge white-blue intersections has positive measure.

For the white-red pair, one should first find~$\tilde{z}$ such that $\log_{\Lambda}\tilde{\Lambda}_0 + \log_\Lambda\tilde{\gamma}(\tilde{z}) - \log_\Lambda\gamma({z}) \in \ZZ$, etc.

\end{proof} 

\begin{rem}\label{rem:triangle}
Proposition~\ref{prop:biangle} can be generalized rather straightforwardly, via generalizing Lemma~2 of~\cite{GK07}, to the case of arbitrary monodromic attracting polycycles that have only hyperbolic saddles as their singularities. 
\end{rem}

%%%%%%%%%%%%%%%%%%%%%%%%%%%%%%%%%%%%%%%%%%%%%%%%%%%%%%%%%%%%
\section{Statistical attractor of a product of two different modified Bowen examples.}\label{sec:twombe}

In this section we will prove the following generalization of Theorem~\ref{thm:stat}, or, more precisely, of Proposition~\ref{prop:statstruct}. The key difference is that now we are considering the product of two arbitrary modified Bowen examples instead of the square of one.

\begin{thm}\label{thm:twombe}
Let $\varphi$ and $\tilde{\varphi}$ be two planar $C^\infty$-flows with modified Bowen examples, i.e., with attracting monodromic polycycles, $\Gamma$ for~$\varphi$ and~$\tilde{\Gamma}$ for~$\tilde{\varphi}$, that have a saddle $B$ (or~$\tilde{B}$ for~$\tilde{\varphi}$) and a generic saddle-node~$A$ (resp.,~$\tilde{A}$). Let monodromic semi-neighborhoods~$X$ and~$\tilde{X}$ of these polycycles be chosen\footnote{This means that we replace $\varphi$ and~$\tilde{\varphi}$ by their restrictions to these sets.} as phase spaces of~$\varphi$ and~$\tilde{\varphi}$. Then for the product flow~$\Phi = \varphi \times \tilde{\varphi}$ we have
\[A_{stat}(\Phi) = \{(A, \tilde{A}), (A, \tilde{B}), (B, \tilde{A})\},\]
\[A_{Mil}(\Phi) = \left(\Gamma \times \{\tilde{A}\}\right) \cup \left(\{A\} \times \tilde{\Gamma}\right).\]
\end{thm}

\begin{proof}
Let the transversals $\Sigma_A$, $\Sigma_B$ and domains $U_A$, $U_B$ be defined as in the previous sections. We will use tilde to denote the analogous objects for the flow $\tilde{\varphi}$. Let $-\mu, \lambda$ be the eigenvalues of the saddle~$B$ and $-b$ be the negative eigenvalue of the saddle-node~$A$.

In appropriate coordinates on~$\Sigma_A, \, \Sigma_B$, the saddle-node monodromy map~${\Delta_{sn} \colon \Sigma_A \to \Sigma_B}$ and transition time $t_{sn}$ can be written as (see~\cite{K06}):
\[\Delta_{sn}(x) = x^{-a}e^{-1/x}; \quad t_{sn}(x) = \frac{1}{b}\cdot \frac{1}{x}(1 + o(1)) \text{ as } x \to 0,\]
where $a$ is the parameter of the smooth orbital normal form of the saddle-node.
Recall that in any smooth coordinates on the transversals with zeros on the separatrices (in particular, in the coordinates used above for the saddle-node), the saddle monodromy map~$\Delta_s$ and transition time are (\cite{T96} and~\cite{K06} or~\cite{GK07}):
\[\Delta_{s}(x) = const \cdot x^{\mu/\lambda}(1 + o(1)); \quad
t_{s}(x) = \frac{1}{\lambda}\ln\frac{1}{x}(1 + o(1)) \text{ as } x \to 0.\]

Let us calculate the asymptotics of the total Poincar\'e map and rotation time:
\[P(x) = \Delta_s \circ \Delta_{sn}(x) = const\cdot(x^{-a}e^{-1/x})^{\mu/\lambda}(1 + o(1)) = e^{-\mu/(\lambda x)\cdot (1 + o(1))},\]
\[T(x) = t_{sn}(x) + t_s(\Delta_{sn}(x)) = \left(\frac{1}{b} + \frac{1}{\lambda}\right)\frac{1}{x}(1+o(1))  \text{ as } x \to 0.\]

It is worthwhile to rewrite this in the coordinate $\eta = 1/x$:
\begin{equation}\label{eq:asy}
P_\eta(\eta) = const \cdot (\eta^{-a}e^{\eta})^{\mu/\lambda}(1 + o(1)) = e^{\frac{\mu}{\lambda} \eta \cdot (1 + o(1))},
\end{equation}
\begin{equation}\label{eq:asy2}
T_\eta(\eta) = \left(\frac{1}{b} + \frac{1}{\lambda}\right) \cdot \eta \cdot (1+ o(1)),  \text{ as } \eta \to +\infty.
\end{equation}

As above, for $z \in \Sigma_A$, let $T_{k, A}(z)$ be the moment when~${\mathcal O}_+(z)$ returns to~$\Sigma_A$ for the $k$-th time and~$T_{k, B}(z)$ be the moment when it intersects~$\Sigma_B$ after~$T_{k, A}(z)$. Let
\[t_k(z) = T_{k+1, A}(z) - T_{k, A}(z)\]
be the time of the $(k+1)$-th turn and let $t_{k, B}(z)$ be the time spent in~$U_B$ during this turn. We will omit the argument $z$ sometimes.

\begin{prop}\label{prop:osmall}
In the double logarithmic time coordinate $\tau(t) = \ln\ln t$ we have
\[\tau(T_{k+1, A}(z)) - \tau(T_{k, B}(z)) = o(1), \text{ as } k \to +\infty.\]
That is, the $\tau$-time spent in~$U_B$ tends to $0$ as the number of the turn grows.
\end{prop}
\begin{proof}
As one can derive from the asymptotics of the monodromy maps and transition times (or read in Proposition~1 of~\cite{K06}), the ratio of times spent in~$U_B$ and~$U_A$ during a turn tends to a constant. Indeed, if there is a point with coordinate~$x$ at the transversal~$\Sigma_A$, it takes time $t_{sn}(x)$ to get to the point $\Delta_{sn}(x)$ at~$\Sigma_B$ and then it takes $t_s(\Delta_{sn}(x))$ to arrive at~$P(x)$, and so we have, using the asymptotic equalities above,
\[\frac{t_s(\Delta_{sn}(x))}{t_{sn}(x)} = \frac{\frac{1}{\lambda x}(1 + o(1))}{\frac{1}{bx}(1 + o(1))} = \frac{b}{\lambda} + o(1) \text{ as } x \to +0.\]
Then, since $P^k(z) \to 0$ as $k \to +\infty$, we can write
\[\frac{t_{k, B}(z)}{t_{k}(z) - t_{k, B}(z)} = \frac{t_s(\Delta_{sn}(P^k(z)))}{t_{sn}(P^k(z))} \xrightarrow{k \to +\infty} \frac{b}{\lambda}.\]
Since for arbitrary positive $x, y, t$ inequality $y < x$ implies $\frac{t + x}{t + y} < \frac{x}{y}$, we have that
\[\frac{T_{k+1, A}(z)}{T_{k, B}(z)} = \frac{T_{k, A}(z) + t_{k}(z)}{T_{k, A}(z) + t_{k}(z) - t_{k, B}(z)} < \frac{t_{k}(z)}{t_{k}(z) - t_{k, B}(z)} \to \frac{b+\lambda}{\lambda}.\]
Now we just consecutively apply logarithms to our two sequences and observe what we get. First, we have
\[\ln(T_{k+1, A}(z)) - \ln(T_{k, B}(z)) = \ln\frac{T_{k+1, A}(z)}{T_{k, B}(z)} < \ln(const + o(1)) < const.\]
Since $T_{k+1, A}(z)$ grows to infinity with~$k$, we have $\frac{\ln(T_{k+1, A}(z))}{\ln(T_{k, B}(z))} \to 1$ and hence
\[\ln\ln(T_{k+1, A}(z)) - \ln\ln(T_{k, B}(z)) = \ln\frac{\ln(T_{k+1, A}(z))}{\ln(T_{k, B}(z))} \to 0, \text{ as } k \to +\infty,\]
which proves the proposition.
\end{proof}

We will use the shorthand notation 
\[\tau_n(z) = \tau(T_{n+1, A}(z)) = \ln\ln(T_{n+1, A}(z)).\]

The following proposition is auxiliary and will only be used in Proposition~\ref{prop:separ}.
\begin{prop}\label{prop:taueta}
Let $\eta_n(z)$ be the $\eta$-coordinates of the points $P^n(z)$, for $n \in \nn$.
Then we have $\tau_n(z) = \ln\ln\eta_n + o(1)$ as $n \to +\infty$.
\end{prop}
\begin{proof}
By~\eqref{eq:asy2}, we have $t_{n} = T_\eta(\eta_{n}) = c\eta_{n} \cdot (1 + o(1))$ as $n \to +\infty$, so
\[\ln\ln t_{n} = \ln\ln(c \eta_{n} (1 + o(1))) = \ln\ln\eta_n + o(1).\]
From~\eqref{eq:asy} we see that $\frac{\eta_{n+1}}{\eta_n} \to +\infty$, and from the form of $T_\eta$ we deduce that $\frac{t_{n+1}}{t_n} \to +\infty$ as well. It is easy to see that for any sequence $(a_n)_n$ such that $\frac{a_{n+1}}{a_n} \to +\infty$ one has $\sum_1^n a_k = a_n(1 + o(1))$ as $n \to +\infty$. This means that
\[T_{n+1, A} = \sum_{k = 0}^n t_k = t_n\cdot(1 + o(1)),\]
which yields
\[\tau_n(z) = \ln\ln(T_{n+1, A}(z)) = \ln\ln t_n + o(1) = \ln\ln\eta_n + o(1) \text{ as } n \to +\infty.\]
\end{proof}

\begin{prop}\label{prop:separ}
For any pair of points~$x, y \in \Sigma_A$ with different $\varphi$-orbits there is an integer~$m$ such that the sets $\{\tau_n(x)\}_{n > m}$ and $\{\tau_n(y)\}_{n > m}$ are separated by some positive distance.
\end{prop}
\begin{proof}
First let us rewrite $P_\eta$ in the double logarithmic coordinate $\zeta = \ln\ln\eta$ and prove that the resulting map is expanding on some domain~$(\zeta_0, +\infty)$. Denote the Poincar\'e map in coordinate~$\zeta$ by~$\Delta(\zeta)$. Then we have
\[\Delta(\zeta) = \ln\ln\left(1/P(1/e^{e^\zeta})\right) = 
\ln\left(-\ln\left(\Delta_s\circ\Delta_{sn}(1/e^{e^\zeta})\right)\right).\]
Let us prove that $\Delta$ is expanding. Applying the chain rule, we have
\[\Delta'(\zeta) = \frac{\left(-\ln(P(1/e^{e^\zeta}))\right)'}{-\ln\left(P(1/e^{e^\zeta})\right)} =
\frac{-\Delta'_s(\Xi) \cdot \Delta'_{sn}(1/e^{e^\zeta}) \cdot \frac{-e^\zeta}{e^{e^\zeta}}}{\Delta_s(\Xi) \cdot (-\ln\Delta_s(\Xi))},\]
where $\Xi$ is a shorthand for $\Delta_{sn}(1/e^{e^\zeta})$, and so $P(1/e^{e^z}) = \Delta_s(\Xi)$.

Recall that $\Delta_{sn}(x) = x^{-a}e^{-1/x}$. Then
$\Delta'_{sn}(x) = e^{-1/x}(x^{-a-2} - ax^{-a-1})$ and
\[\Delta'_{sn}(1/\eta) = e^{-\eta}(\eta^{a+2} - a\eta^{a+1}).\]
Also, relation between $\eta$ and $\zeta$ yields that
$\Xi = \Delta_{sn}(1/\eta) = \eta^{a}e^{-\eta}.$

From~\eqref{eq:asy} we have
\[-\ln P(1/e^{e^\zeta}) = \ln P_\eta(\eta) \sim \frac{\mu}{\lambda}\eta \text{ as } \eta, \zeta \to +\infty.\]

Consider the ratio $\frac{\Delta'_s(\Xi)}{\Delta_s(\Xi)}$. As we discussed earlier in Section~\ref{sec:loop}, we have $\lim_{x \to 0} \frac{x\Delta'_s(x)}{\Delta_s(x)} = \frac{\mu}{\lambda}$, which allows us to write
\[\frac{\Delta'_s(\Xi)}{\Delta_s(\Xi)} \sim \frac{\mu}{\lambda \cdot \Xi} = \frac{\mu}{\lambda \cdot \eta^{a}e^{-\eta}}.\]

Bringing everything together, we have
\[\Delta'(\zeta) \sim \frac{\mu}{\lambda\eta^{a}e^{-\eta}} \cdot \frac{e^{-\eta}(\eta^{a+2} - a\eta^{a+1}) \cdot \frac{e^\zeta}{\eta}}{(\mu/\lambda)\eta} \sim e^\zeta \text{ as } \zeta \to +\infty.\]
So, $\Delta$ is expanding.

Replacing $y$ by another point of its orbit we assume that $y$ is between $x$ and $P(x)$. Let $\zeta_n(x) = \Delta^n(\zeta(x))$ and $\zeta_n(y) = \Delta^n(\zeta(y))$. Since $\Delta$ is monotonic, we always have $\zeta_n(x) < \zeta_n(y) < \zeta_{n+1}(x)$, and since it is expanding, the sequences are separated. But by Proposition~\ref{prop:taueta} we have $\tau_n(z) = \zeta_{n}(z) + o(1)$, so the sequences $(\tau_n(x))_n$ and $(\tau_n(y))_n$ become separated after we drop several elements in the beginning that may coincide.
\end{proof}

\begin{prop}\label{prop:recurr}
There is a recurrence relation
$$\tau_{n+1}(z) = e^{\tau_n(z)} + \ln\frac{b\mu}{b + \lambda} + o(1) \text{ as } n\to +\infty.$$ Notably, the relation has the same form for all~$z$, although the remainder terms of class~$o(1)$ can be different. 
\end{prop}
\begin{proof}
Since $T_\eta(\eta) = c\eta(1+o(1))$ as $\eta \to +\infty$, we can write, using asymptotic equality~\eqref{eq:asy},
\[T_\eta(P_\eta(\eta)) = cP_\eta(\eta)\cdot(1+o(1)) = cP_\eta\left(\frac{1}{c}T_\eta(\eta)(1 + o(1))\right) = e^{\frac{\mu}{c\lambda}T_\eta(\eta)(1+ o(1))} \text{ as } \eta \to +\infty.\]
Hence
\[t_{k+1} = e^{\frac{\mu}{c\lambda}t_{k}(1+ o(1))} \text{ as } k \to +\infty.\]

Since $T_{k+1, A} \sim t_k$ as $k \to \infty$, we can write
\[T_{k+1, A} = e^{\frac{\mu}{c\lambda}T_{k, A} \cdot (1+ o(1))} \text{ as } k\to+\infty.\]
Switching to double logarithms, we rewrite this as
\[\tau_{k+1}(z) = e^{\tau_k(z)} + \ln\frac{\mu}{c\lambda} + o(1).\]
But $c = \frac{1}{b} + \frac{1}{\lambda}$, and we are done.
\end{proof}

Now we are ready to finish the proof of the theorem.

\medskip

{\it End of the proof of Theorem~\ref{thm:twombe}.}
Fix some $z \in \Sigma_A$ as above and denote $a_n = \tau_n(z)$. Take a point $\tilde{z} \in {\Sigma}_{\tilde{A}}$ for another polycycle and denote the analogous sequence by~$(b_n)_n$. Our goal is to prove that for almost every (w.r.t. the product of Lebesgue measures on the transversals) pair~$(z, \tilde{z})$ the positive semi-orbit $\left(\Phi^t(z, \tilde{z})\right)_{t \ge 0}$ visits the set $U_B \times U_{\tilde{B}}$ only finitely many times. After we do this, the proof is finished exactly as in Proposition~\ref{prop:statstruct}.

By Proposition~\ref{prop:osmall}, it suffices to prove that for almost any pair $(z, \tilde{z})$ there is an integer~$m$ such that the sets $\{a_n\}_{n > m}$ and $\{b_n\}_{n > m}$ are separated by positive distance: $\tau$-intervals that correspond to orbits visiting~$U$ or~$\tilde{U}$ get small and get close to the two sequences, so only finitely many of them can intersect. We will show that for a fixed~$\tilde{z}$ this holds for almost every~$z$ and then apply the Fubini theorem.

There are two possibilities for the two sequences $(a_n)_n$ and $(b_n)_n$:
\begin{enumerate}[label=\arabic*)]\setlength\itemsep{-0.1em}
\item either
\[\lim_{k\to +\infty} \inf_n \dist(a_k, b_n) = 0 \;\; \text{ and } \lim_{k\to +\infty} \inf_n \dist(b_k, a_n) = 0,\]
\item or this is not the case.
\end{enumerate}

If the first one holds, replace $z$ with a point $z'$ of different $P$-orbit; the sequence $(a_n)_n$ is then replaced by a corresponding $(c_n)_n$ that (when viewed as a set) after truncation is separated from $\{a_n\}_n$ by Proposition~\ref{prop:separ}. But then for some $m$ the sets $\{c_n\}_{n > m}$ and $\{b_n\}_{n > m}$ are separated: otherwise there would be elements~$c_n$, with large~$n$, close to some $b_k$ and, by our assumption that the first case holds, also close to some~$a_l$, which is impossible, because~$\{a_n\}_n$ and~$\{c_n\}_n$ are separated. Since this holds for all $z'$ except those at the $P$-orbit of $z$, we are done.

Consider the second possibility. In this case there are arbitrary large indices $k$ for which, say, $a_k$ (for $b_k$ the argument would be analogous) is $\e$-far from $\{b_n\}_n$. Let $a_k$ be between $b_n$ and $b_{n+1}$. Recall that by Proposition~\ref{prop:recurr} we have $b_{n+1} = \exp(b_n) + c_2 + o(1)$ as $n \to +\infty$ and $a_{k+1} = \exp(a_k) + c_1 + o(1)$ as $k \to +\infty$. Then, since $n$ goes to infinity with $k$, we have
\[a_{k+1} - b_{n+1} = \exp(a_k) - \exp(b_n) + c_1 - c_2 + o(1).\]
Obviously, $\exp(a_k) - \exp(b_n) > \exp(b_n)(a_k - b_n) > \exp(b_n)\cdot \e$, but we could have chosen $k$ arbitrarily large, so that $n$ is large as well and $\exp(b_n)\cdot \e$ is much greater than any given constant. This allows us to write an estimate
\[a_{k+1} - b_{n+1} > 1.\]
Analogously we obtain the estimate $b_{n+2} - a_{k+1} > 1$ and then argue inductively replacing $k, n$ with $k+1, n+1$ and so on. We get that in this case the sequences separate by unit distance after the pair $k, n$ and never come close to each other again. But then we are done for this particular pair $(z, \tilde{z})$.

For a fixed $\tilde{z}$, either we have the second case for every $(z, \tilde{z})$ and we are done, or we have the first case for some $(z, \tilde{z})$ --- then we replace $z$ with~$z'$ taken from another orbit and obtain separation of sequences as above. Thus we prove that the tails of sequences $(a_n)_n$ and $(b_n)_n$ are separated by positive distance for almost every pair~$(z, \tilde{z})$. As we wrote above, now the proof can be finished as in Proposition~\ref{prop:statstruct}.

\end{proof}

\begin{rem}\label{rem:cod3min}
If we rewrite the sequence of return times for a point of a biangle (i.e., a sequence~$T_{k, A}(z)$ from Proposition~\ref{prop:biangle}) or another attracting monodromic \emph{hyperbolic} polycycle (see Remark~\ref{rem:triangle}) in double logarithmic time, we will get a sequence that diverges to infinity while the distance between consecutive elements converges to zero. Roughly speaking, asymptotically we get a logarithm of an arithmetic series. At the same time, the rotation times for the modified Bowen example quickly diverge to infinity. Thus, the time interval of the $k$-th rotation of a point of the MBE contains many rotation intervals (with some numbers) of a biangle point if~$k$ is large, which implies that the statistical attractor of a product of an MBE and a biangle (or another hyperbolic polycycle) coincides with the product of attractors.

Combining this with Remark~\ref{rem:triangle}, we conclude that the example of non-coincidence for statistical attractors presented in this section is in a sense minimal for flows: one cannot reproduce the same effect using a product of a flow of codimension at most~3 and a flow of codimension at most~2; both flows must have a non-hyperbolic monodromic polycycle with at least two singularities.
\end{rem}

%%%%%%%%%%%%%%%%%%%%%%%%%%%%%%%%%%%%%%%%%%%%%%%%%%%%%%%%%%%%

\section{Minimal attractors and SRB-measures}\label{sec:srb}
\subsection{The theorem}
In this section we will prove the following theorem.
\begin{thm}\label{thm:srb}
There exists a $C^\infty$-smooth flow $\varphi$ on an annulus such that~$\varphi$ has a global physical and natural measure $\mu_{\varphi}$, its square~$\Phi$ also has a global physical and natural measure $\mu_\Phi$, but
\begin{itemize}
\item $\mu_\Phi \neq \mu_{\varphi} \times \mu_{\varphi},$
\item $A_{min}(\Phi) \subsetneqq A_{min}(\varphi) \times A_{min}(\varphi)$.
\end{itemize}
\end{thm}

Although formally the first statement of the theorem does not imply the second, we will have that the supports of~$\mu_\Phi$ and $\mu_{\varphi} \times \mu_{\varphi}$ are different, and it is known (see~\cite[Theorem~1]{GI96}) that when a map or a flow has a global natural measure, its support coincides with the minimal attractor, which yields the second statement.

We emphasize that we construct a single example of a flow that demonstrates our type of non-coincidence for the minimal attractors (and physical measures), or, at best, a very narrow class obtained by varying the parameters of the construction. In Section~\ref{sect:global} we will explain how to adapt this example to an arbitrary surface or higher-dimensional manifold. 

%%%%%%%%%%%%%%%%%%%%%%%%%%%%%%%%%%%%%%%%%%%%%%%%%%%%%%%%%%%%%%%%%
\subsection{Oscillating measures}

There is a trivial example of dynamics on a measure space for which the minimal attractor of the square is smaller than the square of the minimal attractor. Namely, let the phase space be the two-point set~$\{0, 1\}$; endow it with a $\delta$-measure at zero and let the dynamics simply switch the two points. It is easy to see that the minimal attractor of this system is the whole phase space while for the square map the set $\{(0, 1), (1, 0)\}$ always carries zero measure, and so is not a subset of the minimal attractor, and hence the latter is $\{(0, 0),\; (1, 1)\}$.

The idea is to realize something similar in smooth systems on manifolds. We are looking for a flow~$\varphi$ for which the image of the Lebesgue measure under~$\varphi^t$ oscillates, informally speaking, between two domains in the phase space: at some time the whole measure except a very tiny bit concentrates in the `left' part of the phase space; then this measure goes to the `right' part, leaving almost no measure on the left; then the measure comes back to the left, and so on (like in the trivial example above where the $\delta$-measure jumps from one point to another under the dynamics). The `tiny bit' of measure tends to zero with time, of course. In our example we will also have that the measure oscillates, in a sense, between two $\delta$-measures, but the `period' of `oscillations' will tend to infinity.

It seems that we are pointing at an interesting phenomenon here, but it is yet to be seen which formal definition properly captures its essence. In this paper we stick to the first definition that comes to mind.

\begin{defin}
We say that a flow $\varphi$ has \emph{oscillating measure} (or, more precisely, that the Lebesgue\footnote{As usual, we are assuming that our measure is finite and normalized to total measure one.} measure~$Leb$ is oscillating for the flow) if there are two disjoint domains $U_L, U_R$ in the phase space and two increasing sequences of moments of time $t = L_n$ and $t = R_n$, where $L_n, R_n \to +\infty$ as $n \to +\infty$, such that we have
$$(\varphi^{L_n}_*Leb)(U_L) \to 1, \quad (\varphi^{R_n}_*Leb)(U_R) \to 1 \quad \mbox{ as } n \to +\infty.$$
\end{defin}

\begin{rem}
The flow $\varphi$ that is constructed below has oscillating measure.
\end{rem}

\subsection{Outline of the construction}

Consider a finite cylinder $\Cyc = S^1 \times [-1, 1]$. We will assume, for the convenience of description, that in this direct product the circle is horizontal and the segment is vertical. We fix a Lebesgue probability measure~$\mu$ on~$\Cyc$ given by the form $\frac{1}{4\pi}d\t\wedge dz$, where $\t$ is an angle coordinate on the circle and~$z$ is the natural coordinate on the vertical segment. Our plan is to construct a flow on the open cylinder~$\Cyf = S^1 \times (-1, 1)$ and then extend the flow to the cylinder $\Cyc$. The lower boundary circle $S^1 \times \{-1\}$ will consist of fixed points of the extended flow. The closed cylinder can be smoothly embedded into any surface, and our flow can be extended to the whole surface if need be; see Section~\ref{sect:global} below.

The flow on $\Cyf$ is given by a vector field~$V$. The vertical component of~$V$ will be everywhere negative and will be constant on horizontal circles~$S^1 \times \{z\}$, so it can be specified by a function of the vertical coordinate.

\begin{figure}
\center{\includegraphics[width=1\linewidth]{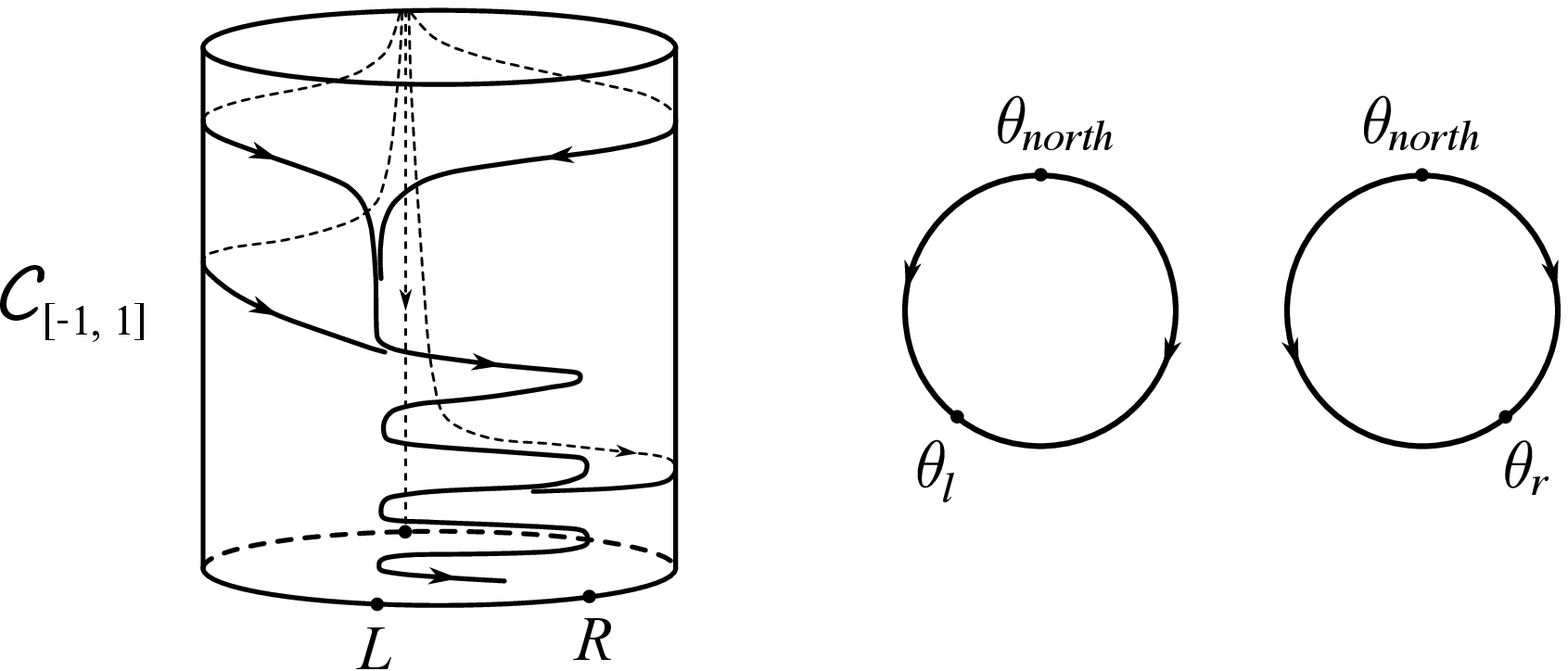}}
\caption{The phase portrait of the flow $\varphi$ on $\Cyc$ and phase portraits of vector fields~$w_{\t_l}\ddt$ and~$w_{\t_r}\ddt$ on~$S^1$.}\label{fig:cyl}
\end{figure}

To specify the horizontal component of~$V$, we consider a smooth one-parameter family~$(w_\alpha\ddt)_{\alpha \in I}$ (here the parameter space~$I$ is a segment) of vector fields on the circle with the following properties.
\begin{enumerate}
\item Each field is north-south, i.e., it has a single hyperbolic sink and a single hyperbolic source and every other point goes from the source to the sink.
\item \label{pr:2} The sources of the fields coincide, have angular coordinate~$\t_{north}$, and are contained in a repelling `northern' domain $N \subset S^1$.
\item The sink $s_\alpha$ of the field $w_\alpha\ddt$ has the angle coordinate $\t = \alpha$ (so, it is natural to identify the segment~$I$ with an arc of~$S^1$ that does not intersect the northern domain~$N$).
\item \label{pr:4} The sinks of the family are uniformly attracting in the complement of the northern domain: $(w_\alpha)'(\t) < -\kappa$, where $\t \in S^1 \setminus N$ is arbitrary and the constant~$\kappa > 0$ is independent of~$\alpha$.
\end{enumerate}

Consider a vertical coordinate~$\eta$ on the open cylinder~$\Cyf$ defined by the equality $\eta = \tan{\frac{\pi z}{2}}$. In a sense, this coordinate turns our finite open cylinder into an infinite cylinder. We will choose a smooth function $h \colon \RR \to I$ and a smooth function~$\rho$ of the vertical coordinate in terms of which the vector field~$V$ on $\Cyf$ will be written as
\begin{equation}\label{eq:v}
V(\t, \eta) = \rho(\eta)\cdot\left(w_{h(\eta)}\ddt - \frac{\partial}{\partial\eta}\right).
\end{equation}

One may think of it this way: we first construct a vector field with unit speed of descent and horizontal component governed by the function~$h$ and then multiply it by some mysterious function~$\rho$. The purpose of this multiplication by~$\rho$ is to make the field smoothly extendable to the lower boundary circle of~$\Cyc$. Note that since the vertical component of the vector field depends only on the vertical coordinate our example is a skew product.

In order to define the function $h$, first fix two points\footnote{In what follows, we will identify points of the circle and their angular coordinates.} $\t_l, \t_r$ outside of the northern domain~$N$ of the circle. We will refer to these as the left point and the right point respectively, thus the subscripts.

Our smooth function $h(\eta)$ must have the following properties.
\begin{enumerate}
\item For $\eta \ge 0$ we have $h(\eta) = \t_l$.
\item In the domain $\{\eta < 0\}$, there are segments separated by unit distance (in terms of~$\eta$) on which $h$ is constant and equal to~$\t_l$ or~$\t_r$. The segments where~$h$ equals~$\t_l$ and $\t_r$ alternate; odd segments correspond to the left point~$\t_l$.
\item\label{pr:h3} In-between these segments our function is monotone. It is the same on any such interim segment~$J$: we take some fixed bump function $b$ supported on a segment $[0, 1]$, consider the function $g(\eta) = \int_0^\eta b(s)ds$, and set $h|_J$ to be equal to this function multiplied by an appropriate constant and shifted appropriately. In particular, this guarantees that all derivatives of $h$ are bounded (but maybe by different constants).
\end{enumerate}

At this point the outline of our construction is given. The only freedom we now have with respect to~$h$ is in how we choose the lengths of the intervals where~$h$ is constant. However, we said nothing whatsoever about how to choose these intervals to make the construction work, that is, to make it have the required properties for the attractors and $SRB$-measures.

\subsection{Plan of the proof}

Now, when the outline of the construction is given, we should first explain what we are going to prove precisely.

One of our goals is to make sure that the flow~$\varphi$ on~$\Cyc$ has a global physical measure. This is going to be the measure $\mu_\varphi = \frac{1}{2}(\delta_{L} + \delta_{R})$ where $L = (\theta_l, \;-1)$ and $R = (\theta_r, \; -1)$.\footnote{Here the point $(\t_r, -1) \in S^1 \times [-1, 1] = \Cyc$ is at the lower boundary of the cylinder. By default we will use the original coordinates $(\t, z)$ on $\Cyc$ to specify points and subsets of the cylinder.}
The square flow~$\Phi$ will have a physical measure
$$\mu_\Phi = \frac{1}{2}(\delta_{(L, L)} + \delta_{(R, R)}).$$
As we mentioned above, a physical measure is also natural, if the phase space is compact. The support of $\mu_\Phi$ consists of two points while the support of $\mu_\varphi \times \mu_\varphi$ has four points, which yields both claims of Theorem~\ref{thm:srb}.

In order to show that $\mu_\varphi$ is indeed the global physical measure for~$\varphi$, we need to check that for any continuous function~$f \colon \Cyc \to \RR$ for Lebesgue-a.e. point $p \in \Cyc$ we have
\begin{equation}\label{eq:srbity}
\lim_{T \to +\infty} \frac{1}{T}\int_{0}^{T}f(\varphi^t(p))dt = \frac{1}{2}(f(L) + f(R)).
\end{equation}
The equality will hold, in fact, for any~$p$ that is not at the union of the two boundary circles and the vertical line $\{\t_{north}\} \times [-1, 1]$ where the sources of the horizontal component live. 
\begin{claim}
To prove that $\mu_\varphi$ is the global physical measure for~$\varphi$, it suffices to show that for any small $\e > 0$ the fraction of time that the positive orbit of~$p$ spends in the stripe ${[\t_l - \e, \t_l + \e]} \times {[-1, 1]}$ tends to $\frac{1}{2}$ and, likewise, the fraction of time it spends in ${[\t_r - \e, \t_r + \e]} \times [-1, 1]$ tends to $\frac{1}{2}$.
\end{claim}
\begin{proof}[Proof of the claim]
Indeed, suppose that the fractions tend to~$\frac{1}{2}$ as time increases to $+\infty$. Given $\delta > 0$ we can choose $\e > 0$ so that for $q \in \Cyc$ we have $|f(q) - f(L)| < \delta$ when $\dist(q, L) < \e$ and also $|f(q) - f(R)| < \delta$ when $\dist(q, R) < \e$. Note that the limit on the left in~\eqref{eq:srbity} will not change if we replace~$p$ with another point of the same orbit, because this would be equivalent to shifting the lower and the upper integration limits by the same constant. So, let us replace~$p$ with its image whose vertical coordinate is close to $-1$. Then take some large~$T$ and split the integral from~$0$ to~$T$ into integrals over intervals where $\varphi^t(p)$ is $\e$-close to~$L$, intervals where $\varphi^t(p)$ is $\e$-close to~$R$, and intervals where $\varphi^t(p)$ is far from both~$L$ and~$R$. By our assumption, if $T$ is large, the intervals of the third type make a small fraction of the segment $[0, T]$ while the intervals of the first type have total length $\frac{1}{2}T(1+o(1))$ and the same holds for intervals of the second type. At the intervals of the first type the integrand is $\delta$-close to $f(L)$, at the intervals of the second type it is $\delta$-close to $f(R)$, and at the intervals of the third type it is bounded by the maximal value of $f$ on $\Cyc$. This yields an estimate of the form
$$\left|\frac{1}{T}\int_{0}^{T}f(\varphi^t(p))dt - \frac{1}{2}(f(L) + f(R))\right| < const \cdot \delta.$$
Hence, the limit in~\eqref{eq:srbity} exists and the equality itself holds.
\end{proof}
An analogous claim can be made about the physical measure of the square flow~$\Phi$ instead of the flow~$\varphi$.
  
\subsection{Intervals and descent speed}

We are now to choose the intervals where $h(\eta)$ is constant and also choose the function~$\rho$ of the vertical coordinate.

Let $h$ be equal to $\t_l$ on the first segment, which we denote by~$I_1$. Then let $I_2$ be the union of the interim unit segment where the value of~$h$ changes from~$\t_l$ to~$\t_r$ and the adjacent segment where it equals~$\t_r$. The rest of the segments~$I_n$ are defined analogously; adjacent segments~$I_n, I_{n+1}$ intersect only by their endpoints. The segments form a partition of the $z$-interval $(-1, 0]$; in terms of~$\eta$ this is interval~$\{\eta \le 0\}$.

To define the lengths of the segments $I_n$, we specify the vertical coordinate~$\xi$ on $(-1, 0)$ in which they all look like segments of length one: $\xi = \sqrt{-\eta} = \sqrt{-\tan{\frac{\pi z}{2}}}.$ Note that the $\xi$-axis is directed downwards.

Now we specify the speed of descent in the same coordinate~$\xi$. More formally,  we denote by~$\sigma(\xi)$ the $\frac{\partial}{\partial\xi}$-component of the vector field~$V$ written in the $(\frac{\partial}{\partial\t}, \frac{\partial}{\partial\xi})$ basis. Consider an auxiliary function $g(\xi) = e^{\sqrt{\xi}}$. We define $\sigma$ by the following equality:
$$\sigma(\xi) = \frac{1}{g'(\xi)}.$$

Note that by specifying $\sigma(\xi)$ we also define $\rho(\eta)$. In our skew product system, the vertical movement is governed by an equation $\dot\xi = \sigma(\xi)$ in the~$\xi$-coordinate. Switching back to coordinate~$\eta = -\xi^2$ we get
$$\dot \eta = -2\xi\dot\xi = -2\xi\sigma(\xi) = -2\sqrt{-\eta} \cdot \sigma(\sqrt{-\eta}) = -4(-\eta)^{3/4}\cdot e^{-\sqrt[4]{-\eta}}.$$

But according to~\eqref{eq:v} the right hand side here must be equal to $-\rho(\eta)$, so we get
\begin{equation}\label{eq:expformula}
\rho(\eta) = 4(-\eta)^{3/4}\cdot e^{-\sqrt[4]{-\eta}}.
\end{equation}

Unfortunately, this function is not smooth at zero, so let us modify it at the $\eta$-interval $(-1, 0)$ in such a way that it becomes equal to~$1$ and flat at zero and so can be smoothly extended by the constant~$1$ to the domain $\eta > 0$. To simplify calculations in the following section, we will assume that the function~$\sigma(\xi)$ is modified on~$(0,1)$ accordingly and, moreover, that this modification leaves the integral $\int_0^1 \frac{1}{\sigma(\xi)} d\xi$ unchanged (i.e., equal to~$e - 1$).

At this point the flow $\varphi$ is finally defined. It is clear from the first property of function~$h(\eta)$ (this function is constant for~$\eta > 0$) that the vector field extends smoothly to the upper boundary circle of the cylinder~$\Cyc$. However, we are yet to check that it extends smoothly to the lower circle. We will do this in Section~\ref{sec:smooth}.

\subsection{Physical measure of \texorpdfstring{$\varphi$}{phi}}
\label{sec:srbphi}
Consider an arbitrary point of~$\Cyf$ with angular coordinate different from $\t_{north}$. Denote by $p$ the point of the same orbit at the zero-circle $S^1 \times \{0\}$ and set $p(t) = \varphi^t(p)$. By the \emph{claim} above, it suffices to estimate the difference between $\frac{1}{2}$ and the fraction of time that $p(t)$ spends in a thin vertical stripe that contains the point~$L$ or the point~$R$.

Fix some small $\e > 0$ and denote the stripe ${[\t_l - \e, \t_l + \e]} \times {[-1, 1]}$ by~$S_L$ and the stripe ${[\t_r - \e, \t_r + \e]} \times {[-1, 1]}$ by~$S_R$.

Let us call the sets of the form $S^1 \times I_n$ \emph{blocks}. For the $n$-th block with odd~$n$, denote by~$\alpha_n$ the following ratio: the numerator is the time our orbit spends outside of the stripe~$S_L$ while passing through the block, and in the denominator we have the whole time it takes for our orbit to cross the block. If~$n$ is even, $\alpha_n$ is defined analogously with~$S_L$ replaced by~$S_R$.

\begin{prop}\label{prop:alpha}
We have $\alpha_n \to 0$ as $n \to +\infty$.
\end{prop}
\begin{proof} It may be that the point $p$ is in the northern stripe~$N \times [-1, 1]$ and it takes very long time before it gets out of there, but this happens eventually and as soon as it does the point never goes back to~$N \times [-1, 1]$.

Consider an interval $I_n$ with~$n$ so large that, when~$p(t)$ gets to the corresponding block, it is already out of the northern stripe. Each~$I_n$ with $n > 1$ starts with an interim segment that has length one in coordinate~$\eta$. When the orbit of~$p$ enters the $n$-th block, it takes some time first to cross the cylinder that corresponds to the interim interval. When this is done, the value of function~$h$ becomes equal to $\t_l$ or~$\t_r$ for the rest of the block. Assume that~$n$ is odd and the value is equal to~$\t_l$. Then the horizontal component of the field has a sink at~$\t_l$ and so the orbit of $p$ is attracted to the stripe~$S_L$. It takes some time for the orbit of $p$ to get into the stripe~$S_L$. Recall that in the $(\t, \eta)$-coordinates the vector field~$V$ has the form~\eqref{eq:v}. It is easier to first write an estimate on the vertical segment the point passes before it gets into~$S_L$. For this purpose, let us divide the vector field in~\eqref{eq:v} by~$\rho(\eta)$ and consider
\[\hat{V} = w_{h(\eta)}\ddt - \frac{\partial}{\partial\eta}.\]
For this vector field we have unit speed of descent and horizontal component~$w_{h(\eta)}\ddt$. Property~\ref{pr:4} of the family $(w_\alpha\ddt)_\alpha$ implies that the time it takes for a point of a circle that starts outside of the northern domain to get $\e$-close to the horizontal sink at~$\t_l$ is at most~$\frac{1}{\kappa}\log\frac{2\pi}{\e}$, but then we have the same estimate on the length of the vertical segment the point crosses before its orbit gets into~$S_L$. The exact form of this estimate is not important: it is just a number that depends on~$\e$ but not on~$n$. The important part is that the estimate on the length of the vertical segment stays valid for the field~$V$. Denote by~$J_n \subset I_n$ the upper part of the segment~$I_n$ that ends when $p(t)$ finally gets into~$S_L$. 

Now, the $\xi$-coordinates of the endpoints of~$I_n$ are~$n-1$ and~$n$. The corresponding $\eta$-coordinates are~$-(n-1)^2$ and~$-n^2$, so the length of this segment in the $\eta$-coordinate is $2n-1$, which grows linearly with~$n$. The ratio between the length of the vertical segment $J_n$ (its length is at most~$1 + \frac{1}{\kappa}\log\frac{2\pi}{\e}$) and the length of~$I_n$ converges to zero as $n \to +\infty$.

We are interested, however, not in the ratio of lengths of vertical segments but in the ratio of times it takes for our orbit to cross them. For the vector field~$V$, the speed of vertical movement in the $\eta$-coordinate is $-\rho(\eta)$, and the value of $\rho$ decreases monotonically to zero for large values of~$\eta$, so the absolute speed is always greater in the upper part of the cylinder than in the lower; therefore the ratio of times is even smaller than the ratio of lengths, and it also converges to zero.

The same argument works for even~$n$.
\end{proof}

Let $t_n$ be the moment of time when~$p(t_n)$ leaves the $n$-th block and let~$t_0 = 0$. Denote by $\chi_l(n)$ (resp.~$\chi_r(n)$) the fraction of the segment of time~$[t_0, t_n]$ that~$p(t)$ spends in the stripe~$S_L$ (resp.~$S_R$). We will prove that the difference $\chi_r(n) - \chi_l(n)$ converges to zero. Proposition~\ref{prop:alpha} implies that $\chi_r(n) + \chi_l(n) \to 1$, so convergence of the difference to zero will yield that $\chi_l(n) \to \frac{1}{2}$ and $\chi_r(n) \to \frac{1}{2}$, which is what we want to prove.

First let us find $t_n$. Since the vertical movement is governed by equation $\dot\xi = \sigma(\xi)$ and the moment $t_n$ corresponds to the $\xi$-coordinate being equal to~$n$, we have
$$t_n = \int_0^n \frac{1}{\sigma(\xi)} d\xi = e^{\sqrt{n}} - 1.$$

Recall that, for the $k$-th block, our orbit spends time $(t_k - t_{k-1})\cdot(1 - \alpha_k)$ in~$S_L$ if~$k$ is odd or in~$S_R$ if~$k$ is even. Then we have that
$$\chi_r(n) - \chi_l(n) = \frac{1}{t_n - t_0}\sum_{k=1}^n (-1)^k \cdot (t_k - t_{k-1})\cdot(1 - \alpha_k) = $$
$$= \underbrace{\frac{1}{t_n} \sum_{k=1}^n (-1)^k \cdot (t_k - t_{k-1})}_{E_1}
 + \underbrace{\frac{1}{t_n} \sum_{k=1}^n \alpha_k \cdot (-1)^{k+1} \cdot (t_k - t_{k-1})}_{E_2}.$$

Consider expression~$E_1$ first. Since the difference $t_k - t_{k-1} = e^{\sqrt{k}} -e^{\sqrt{k-1}}$ grows with~$k$, the value of~$E_1$ is negative for odd~$n$ and positive for even~$n$. Moreover, the sum in~$E_1$ does not exceed its last term in absolute value, so we have
$$|E_1| \le \frac{t_n - t_{n-1}}{t_n} = \frac{e^{\sqrt{n}} - e^{\sqrt{n-1}}}{e^{\sqrt{n}} - 1} = \frac{1 - e^{\sqrt{n-1} - \sqrt{n}}}{1 - e^{-\sqrt{n}}} = \frac{1 - e^{\frac{-1}{\sqrt{n-1} + \sqrt{n}}}}{1 - e^{-\sqrt{n}}} \to 0 \quad (n \to +\infty).$$

Now consider~$E_2$. Recall that~$\alpha_k \to 0$. Fix some small~$\delta > 0$ and split the sum in~$E_2$ into a sum over~$k$ for which $\alpha_k \ge \delta$ and a sum for which~$\alpha_k < \delta$. The first sum is bounded and we divide it by $t_n$, so the ratio converges to zero as~$n \to +\infty$. The modulus of the second sum is smaller than~$\delta\sum_0^n(t_k - t_{k-1}) = \delta t_n$, and so the ratio with~$t_n$ is smaller than~$\delta$. Since~$\delta$ was arbitrary, we get that $E_2 \to 0$ as $n \to +\infty$.

We have established that $\chi_r(n) - \chi_l(n) \to 0$. For a given $\xi > 0$ let $\chi_r(\xi)$ be the fraction of the time interval~$[t_0, t_\xi]$ that~$p(t)$ spends in~$S_R$, where~$t_\xi$ is the moment when $p(t_\xi)$ has $\xi$-coordinate equal to~$\xi$. Define~$\chi_l(\xi)$ analogously. It is straightforward to show that~$|\chi_r(\xi) - \chi_r([\xi])| \to 0$ as~$\xi \to +\infty$ (or, equivalently, as~$t_\xi \to +\infty$) and that the same holds for~$\chi_l(\xi)$. This implies that both~$\chi_r(\xi)$ and~$\chi_l(\xi)$ converge to~$\frac{1}{2}$, but by the \emph{claim} this proves that~$\mu_\varphi$ is the physical measure for~$\varphi$.

%%%%%%%%%%%%%%%%%%%%%%%%%%%%%%%%%%%%%%%%%%%%%

\subsection{Physical measure of \texorpdfstring{$\Phi$}{Phi}}
\label{sec:srbPhi}

In this section we prove that the measure~$\mu_\Phi = \frac{1}{2}(\delta_{(L, L)} + \delta_{(R, R)})$ is physical for the flow~$\Phi$. In order to do that, it suffices to prove the following proposition, where for~$A \subset \Cyc$ we write~$A^2$ as a shorthand for~$A \times A \subset \Cyc \times \Cyc$. 

\begin{prop}
For a Lebesgue-a.e. pair~$(p, q) \in \Cyf^2$ and arbitrary $\e > 0$, the fraction of time that a segment~$\{\Phi^t(p, q)\}_{t \in [0, T]}$ of the positive orbit has spent in the set
\[S_{LL} = S_L^2 = ([\t_l - \e, \t_l + \e] \times [-1, 1])^2\]
tends to $\frac{1}{2}$ as $T \to +\infty$, and the fraction of time it has spent in
\[S_{RR} = S_R^2 = ({[\t_r - \e, \t_r + \e]} \times [-1, 1])^2\]
tends to $\frac{1}{2}$ as well.
\end{prop}
\begin{proof}
The condition on $(p, q)$ is that neither $p$ nor~$q$ have $\t$-coordinate equal to $\t_{north}$. We may freely replace the pair~$(p, q)$ with~$\Phi^t(p, q)$ for some $t$, so we will assume that the point~$p$ is at the zero-circle $S^1 \times \{0\} \subset \Cyc$. We will also assume WLOG that~$q$ is at this circle or above it.

Denote the orbits of~$p$ and~$q$ by $p(t)$ and~$q(t)$. Switch to the vertical coordinate~$\eta$ and denote the vertical distance in~$\eta$ between $p(t)$ and~$q(t)$ by~$d(t)$. Then~$d(t)$ is a bounded function of~$t > 0$. Indeed, the case when~$p$ and~$q$ start at the same vertical level is trivial. If~$q$ is higher than~$p$, we argue that the speed of vertical descent of a given point depends only on the vertical coordinate and monotonically decreases to zero as $\eta\to -\infty$, starting from some level~$\eta_0$; so after both orbits cross this vertical level, the point~$q(t)$, which is higher and therefore moves downwards faster, begins to catch up with~$p(t)$ in terms of the~$\eta$-coordinate, and~$d(t)$ starts to decrease. Thus, we can assume that~$d(t)$ is bounded by a constant~$D$.

Consider $n$ so large that the $\eta$-length of~$I_n$ is greater than~$D$. Say, $n$ is odd. Denote the $n$-th block by~$B_n$ and let the number~$\beta_n$ be the ratio between the time during which $p(t)$ is in~$B_n$, but $(p(t), q(t)) \notin (B_n \cap S_L)^2$, and the whole time during which $p(t) \in B_n$. If~$n$ is even, define $\beta_n$ analogously, but with~$S_{L}$ replaced by~$S_{R}$.

Arguing as in Section~\ref{sec:srbphi}, we have that $\beta_n \to 0$ as $n \to +\infty$. Indeed, the vertical distance that the first orbit~$p(t)$ crosses between the moment when it enters the~$n$-th block and the moment when both~$p(t)$ and~$q(t)$ are in~$S_L$ is bounded. Thus, the ratio of this distance to the length of the block tends to zero, and the ratio of the corresponding time segments is even smaller. We also know from Section~\ref{sec:srbphi} that the fraction of the time interval~$[0, T]$ that the orbit~$p(t)$ spends in the left stripe~$S_L$ (or right stripe~$S_R$) converges to~$\frac{1}{2}$ as~$T \to +\infty$. Convergence of $\beta_n$ to zero implies then that the fraction of~$[0, T]$ that both~$p(t)$ and~$q(t)$ spend in~$S_L$ (resp.~$S_R$) tends to~$\frac{1}{2}$. But this is the same as saying that~$\Phi^t(p, q)$ asymptotically spends half of the time in~$S_{LL}$ (or~$S_{RR}$), and the proposition is proven.
\end{proof}

%%%%%%%%%%%%%%%%%%%%%%%%%%%%%%%%%%%%%%%%%%%%%

\subsection{The flow \texorpdfstring{$\varphi$}{phi} is smooth}\label{sec:smooth}
In this section we will check that the vector field~$V$ defined in~$\Cyf$ provides a smooth vector field in~$\Cyc$ when it is extended by continuity to boundary circles. As we remarked above, it is trivial that the extended vector field is smooth near the upper boundary circle; so we focus on the lower one.

Consider a new vertical coordinate~$\zeta$ defined in the upper semi-neighborhood of $-1$ in the original $z$-coordinate:
\[\zeta = -\frac{1}{\eta} = \frac{1}{\xi^2} = -\cot \frac{\pi z}{2}.\]
The map $z \mapsto \zeta(z) = -\cot \frac{\pi z}{2}$ is a $C^\infty$-diffeomorphism of a neighborhood of $-1$ onto the image, so $\zeta$ is a well-defined smooth coordinate.

In the coordinates $(\t, \zeta)$ the vector field~$V$ has the form
\[V(\t, \zeta) = \rho(-1/\zeta)\left(w_{h(-1/\zeta)}(\t)\ddt - \zeta^2\frac{\partial}{\partial\zeta}\right).\]

Denote $\trho(\zeta) = \rho(-1/\zeta)$. Then in a semi-neighborhood of zero in the $\zeta$-coordinate we have, by~\eqref{eq:expformula},
\[\trho(\zeta) = 4\zeta^{-3/4}\cdot e^{-\zeta^{-1/4}}.\]
This function is flat at zero together with all its derivatives. Now it is obvious that the vertical component of the vector field~$V(\t, \zeta)$, namely $-\zeta^2\trho(\zeta)$ is smooth in a neighborhood of the lower boundary circle.

The horizontal component is~$\trho(\zeta) \cdot w_{h(-1/\zeta)}(\t)$. For convenience, let us write $f(\t, \zeta)$ instead of $w_{h(-1/\zeta)}(\t).$ Since~$\trho(\zeta)$ is flat at zero and~$f(\t, \zeta)$ is bounded, the horizontal component is flat at the lower boundary circle. However, it is not entirely obvious that it is smooth there. In other words, we are to check that all derivatives of the horizontal component have zero limit at $\zeta = 0$. We present the argument for derivatives in~$\zeta$. The argument for a derivative $\frac{\partial^{n+k}}{\partial\zeta^n\partial\t^k}$ can be obtained by formally replacing every instance of~$f$ by~$\frac{\partial^k{f}}{\partial\t^k}$.

Applying the Leibniz rule we write
\[ \left(\ddzeta\right)^n(\trho \cdot f)= \sum_{j=0}^n \binom{n}{j} \left(\ddzeta\right)^j\trho \cdot \left(\ddzeta\right)^{n-j}f. \]
Here $\left(\ddzeta\right)^j\trho$ is smooth and flat at zero, but $\left(\ddzeta\right)^{n-j}f$ is not bounded near zero. However, the latter factor is majorized by some negative power of~$\zeta$ in a neighborhood of zero. Indeed,
\begin{equation}\label{eq:deriv}
\ddzeta f(\t, \zeta) = \ddzeta(w_{h(-1/\zeta)}(\t)) = (w_\alpha)'_\alpha(\t)|_{\alpha = h(-1/\zeta)} \cdot h'(-1/\zeta) \cdot \frac{1}{\zeta^2}.
\end{equation}
The family $(w_\alpha)_{\alpha \in I}$ is smooth and $I$ is compact, so all possible derivatives of $w_\alpha$ are bounded (by different constants). Likewise, by property~\ref{pr:h3} of~$h$ all derivatives of~$h$ are bounded. So, on the right in~\eqref{eq:deriv} we have a product of two bounded functions and a negative power of~$\zeta$. Moreover, as we differentiate this expression in~$\zeta$ finitely many times to calculate~$\left(\ddzeta\right)^{n-j}f$, we always get a sum of terms each of which is a product of bounded functions and a negative power of~$\zeta$. Hence, $\left(\ddzeta\right)^{n-j}f$ is majorized by some negative power of~$\zeta$ and when multiplying it by a flat~$\left(\ddzeta\right)^j\trho$ we get a function that is flat at zero. This shows simultaneously that the derivatives~$\left(\ddzeta\right)^n(\trho \cdot f)$ are defined when $\zeta = 0$ and that they are continuous at the lower boundary circle. As we mentioned above, the argument for the rest of the derivatives is the same.

At this point the proof of Theorem~\ref{thm:srb} is complete.

\subsection{Remarks on \texorpdfstring{$\rho$}{rho} and \texorpdfstring{$\xi$}{xi}}
The function $\rho$ and the coordinate $\xi$ (in which the segments~$I_n$ have constant length) are rather particular, so it is worthwhile to explain how those have been chosen.

From the coordinate $\xi$ we want only one thing: we want it to be related to coordinate $\eta$ in such a way that the lengths of~$I_n$ monotonically increase to~$+\infty$ when measured in~$\eta$. Then we have that the ratios between $J_n$ and $I_n$  go to zero, which is one of the key points in Sections~\ref{sec:srbphi} and~\ref{sec:srbPhi}.

As for the function~$\rho$, it is a bit trickier. The function~$\rho$ is defined in terms of~$\sigma(\xi)$, so we will discuss the choice of~$\sigma$ instead. The first thing we want from~$\sigma(\xi)$ is that it be decreasing and the integral $\int_{n-1}^n \frac{d\xi}{\sigma(\xi)}$, as a function of~$n$, be of smaller order than $\int_{0}^n \frac{d\xi}{\sigma(\xi)}$ as~$n \to +\infty$. This holds, e.g., for all positive powers of~$\xi$, but does not hold for~$e^\xi$. The sum, with alternating signs, of integrals over the segments~$I_n$ is basically the difference of times spent in the right and in the left vertical strips (see Section~\ref{sec:srbphi}). This sum is less than the last summand in modulus, and the total time spent in the first~$n$ blocks (it is $\int_{0}^n \frac{d\xi}{\sigma(\xi)}$) is asymptotically much larger, which makes the time-average-measures over the trajectory converge to~$\mu_\phi$.

However, there is the second property we need~$\rho(\eta)$ to have: it and its derivatives must decrease to zero faster than any negative power of~$-\eta$ as $\eta \to -\infty$, because this makes the function~$\trho(\zeta) = \rho(-1/\zeta)$ and its derivatives flat, which is crucial for Section~\ref{sec:smooth}.
This naturally led to making~$g$, the anti-derivative of $1/\sigma$, be a function of intermediate growth, such as~$e^{\sqrt{\xi}}$.

\subsection{Generalization to arbitrary manifolds}\label{sect:global}
In this section we will show how to adapt the example above from the cylinder to any closed surface or manifold of higher dimension.

Assume that the cylinder~$\Cyc$ is embedded in $\RR^3$ in a trivial way, i.e., as a product of the unit horizontal circle and the vertical segment~$[-1, 1]$, and glue it with its reflection with respect to its lower boundary circle. The reflection of our vector field then is smoothly glued to the vector field itself, because the latter is smooth and flat at the lower boundary.

In a neighborhood of the upper boundary circle of the original cylinder the vector field can be modified (e.g., by multiplying it by a flat function) so that we will be able to shrink the boundary circle into a point in such a way that this point becomes a source of the vector field. Furthermore, we can cut away the corresponding (w.r.t. the reflection) neighborhood of the lower boundary circle of the reflected cylinder, so that the remainder of the field becomes transverse to the new boundary circle (and directed inwards). Thus we obtain a smooth vector field on a disc, and this vector field has the same properties regarding the SRB-measures as the one on the cylinder had. The argument is the same as above. Note that when considering a pair of points of the phase space as in Section~\ref{sec:srbPhi}, if the points are on different sides of the circle~$S^1 \times \{-1\}$, we can replace the lower of the two points by its image under the reflection and then argue exactly as in Section~\ref{sec:srbPhi} to show that the physical measure of the square flow is supported on two points.

If we want to adapt this to higher dimensions, we can take a direct product of the vector field on the disk with a linear contracting vector field, thus obtaining a vector field on a ball of required dimension (the field is directed inwards at the boundary sphere). Then we take a closed manifold~$M$ (of dimension at least two), consider a gradient vector field of a Morse function that has a single maximum point and replace the vector field in a neighborhood of this point by our vector field on a ball\footnote{The field on the ball constructed above can be rectified near the boundary sphere and so can be glued smoothly to the field on the manifold. For the resulting flow, any set inside a ball that contains almost every its point will intersect almost every orbit of the flow.}. Since there was only one sink for the gradient flow, all trajectories but few went to that sink. They go inside the ball now, and there they get exponentially attracted to the two-dimensional invariant disk, and it is only the dynamics on the disk that is important when we consider their asymptotic behavior. This construction yields the following theorem.

\begin{thm}\label{thm:global}
For any closed $C^\infty$-smooth manifold of dimension at least two there exists a $C^\infty$-smooth flow $\varphi$ such that~$\varphi$ has a global physical and natural measure $\mu_{\varphi}$, its square~$\Phi$ also has a global physical and natural measure $\mu_\Phi$, but
\begin{itemize}
\item $\mu_\Phi \neq \mu_{\varphi} \times \mu_{\varphi},$
\item $A_{min}(\Phi) \subsetneqq A_{min}(\varphi) \times A_{min}(\varphi)$.
\end{itemize}
\end{thm}

\vspace{0.3cm}

\medskip
\noindent
{\large Stanislav Minkov,\\}
Brook Institute of Electronic Control Machines, Moscow,\\
E-mail: \texttt{stanislav.minkov@yandex.ru}

\vspace{0.6cm}

\medskip
\noindent
{\large Ivan Shilin,\\}
National Research University Higher School of Economics, Moscow,\\
Faculty of Mathematics,\\
E-mail: \texttt{i.s.shilin@yandex.ru}

\end{document}